\documentclass[11pt]{amsart}
\usepackage{amsmath,amssymb,enumerate}
\usepackage{epsf,epsfig,amsfonts,graphicx,color,url}

\usepackage[autostyle]{csquotes}

\newcommand{\beq}{\begin{equation}}
\newcommand{\eeq}{\end{equation}}

\usepackage[all]{xy}
\usepackage{filecontents}

\DeclareMathOperator{\Eng}{Eng}

\def\R{\mathbb{R}}

\def\Ri{Riemannian }
\def\on{orthonormal  }

\numberwithin{equation}{section}

\newtheorem{theorem}{Theorem}

\newtheorem{thm}{Theorem}
\newtheorem{proposition}[thm]{Proposition}
\newtheorem{corollary}[thm]{Corollary}
\newtheorem{lemma}[thm]{Lemma}
\newtheorem{definition}[thm]{Definition}
\newtheorem{conj}[thm]{Conjecture}
\newtheorem{theorem1}[thm]{Theorem}

\renewcommand{\emph}[1]{{\bfseries\itshape{#1}}}

\usepackage{hyperref}
\usepackage[%
	capitalize,nameinlink
]{cleveref}
\hypersetup{%
	colorlinks=true,
	linkcolor=blue,
	citecolor=magenta,
	urlcolor=magenta,
}

\numberwithin{figure}{section}

\def\Ri{Riemannian }
\def\SR{Sub-Riemannian }

\def\sR{sub-Riemannian } 
\def\on{orthonormal }

\newcommand{\Di}{\mathcal{D}}

\newcommand{\Je}{J^k(\mathbb{R},\mathbb{R})}

\begin{document}

\newtheorem*{backgroundtheorem}{Background Theorem}


\title[Metric lines in the Jet Space.]{Metric lines  in the  Jet Space.}  
\author[A.\ Bravo-Doddoli]{Alejandro\ Bravo-Doddoli} 
\address{Alejandro Bravo: Department of Mathematic University of Michigan, Ann Arbor, MI 48109, U.S.}
\email{Abravodo@umich.edu}
\keywords{Carnot group, Jet space, Global minimizing geodesic, sub
Riemannian geometry}
\begin{abstract} 
Given a \sR manifold, a relevant question is: what are the metric lines (isometric embedding of the real line)?
The space of $k$-jets of a real function of one real variable $x$, denoted by $\Je$, admits  the structure of a Carnot group, as every Carnot group $\Je$ is a \sR Manifold.
This work is devoted to provide a partial result about the classification of the metric lines in $\Je$.  The method to prove the main Theorems is to use an intermediate $3$-dimensional \sR space $\R^{3}_F$  lying between the group $\Je$ and the Euclidean space $\R^{2} \simeq \Je / [\Je,\Je]$.
\end{abstract}

\maketitle

\section{Introduction}

The space of $k$-jets of a real function of one real variable $x$, denoted by $\Je$, admits  the structure of a Carnot group, as every Carnot group $\Je$ is a \sR Manifold.
This work is the first of a sequence of papers, where we attempt to make a full classification of the metric lines in $\Je$. Let us introduce the definition of a metric line in the context of \sR geometry.
\begin{definition}\label{def:metric-line}
\index{Metric lines} \index{$dist_{M}(\cdot,\cdot)$} Let $M$ be a \sR manifold, we denote by  $dist_{M}(\cdot,\cdot)$ the \sR distance on $M$. Let $|\cdot|:\R \to [0,\infty)$ be the absolute value. We say that a geodesic $\gamma:\R \to M$ is a metric line if  
$$|a-b| = dist_{M}(\gamma(a),\gamma(b))\;\;\; \text{for all compact set}\;\;\; [a,b]\;\; \subset \;\;\R. $$
\end{definition}
See Definition \cite[sub-sub-Chapter 4.7.2 ]{agrachev} or \cite[sub-Chapter 1.4 ]{tour} for the formal definition of a \sR geodesic. An alternative term for \enquote{metric line} is \enquote{globally minimizing geodesic}.

In \cite[Background Theorem]{RM-ABD}, we showed a bijection between the set of pairs $(F,I)$ and the set of geodesics in $\Je$, where $F$ is a polynomial of degree $k$ or less, and $I$ is a closed interval, called the hill interval, see Definition \ref{def:hill-int} below. The polynomial $F$ defines a reduced Hamiltonian system $H_F$, see equation \eqref{eq:fund-eq-jet-space} below, provided by the sympletic reduction of \sR geodesic flow on $\Je$. In addition,  we classified the geodesic  in $\Je$ according to their reduced dynamics, that is, the geodesics in $\Je$ are line, $x$-periodic, homoclinic, heteroclinic of the direct-type or heteroclinic of the turn-back, see sub-Section \ref{sec:clas-geo-je} or Figure \ref{fig:perio-hom-curve}. The Conjecture concerning metric lines in $\Je$ is the following.
\begin{conj}\label{con:1}
The metric lines in $\Je$ are precisely geodesics of the type: line, homoclinic and the heteroclinic of the direct-type.
\end{conj}
It is well know that the line geodesics are metric lines, see Corollary \ref{cor:sR-sub-metric-line}. In \cite[Theorem 1]{RM-ABD}, we proved geodesics of type $x$-periodic and heteroclinic turn-back are not metric lines.  Theorem \ref{th:main-1} is the first main result of this work and proves Conjecture \ref{con:1} for the case of heteroclinic of the direct-type geodesics.
\begin{theorem}\label{th:main-1}
Heteroclinic of the direct-type geodesics are metric lines in $\Je$.
\end{theorem}
Conjecture \ref{con:1} remains open for homoclinic geodesics. Theorem \ref{th:main-2} is the second principal result of this work and provides a family of homoclinic geodesics that are metric lines. 
\begin{theorem}\label{th:main-2}
The homoclinic-geodesic defined by the polynomial $F(x) = \pm(1 - bx^{2n})$ and hill interval $[0,\sqrt[2n]{\frac{2}{b}}]$ is a metric line in $\Je$ for all $k\geq 2n$ and $b>0$.
\end{theorem}

\subsection*{Previous Results} In \cite{ardentov1,ardentov2,ardentov3,ardentov4}, A. Andertov and Y. Sachkov proved Conjecture \ref{con:1} for the case $k=1$ and $k=2$ using optimal synthesis. In \cite[Theorem 2]{RM-ABD}, we showed that a family of heteroclinic of the direct-type geodesics are metric lines.

The case $k=1$ corresponds to $J^1(\mathbb{R},\mathbb{R})$ being the Heisenberg group where the geodesics are $x$-periodic or geodesic lines.
The case $k=2$ corresponds to $J^2(\mathbb{R},\mathbb{R})$ being Engel's group, denoted by $\Eng$. Besides geodesic lines, up to a Carnot translation and dilation $\Eng$ has a unique metric line such that its projection to the plane $\R^{2} \simeq \Eng/[\Eng,\Eng]$ is the Euler-soliton.  The family of metric lines defined by Theorem \ref{th:main-2} is the generalization of A. Andertov and Y. Sachkov's  result from \cite{ardentov1,ardentov2,ardentov3,ardentov4}. More specific, when $n=1$ then the geodesic defined by the polynomial $F(x) = \pm(1 - bx^{2})$ is the one whose projection to the plane $\R^{2}$ is the Euler-soliton. See \cite[Section 4]{ABD} for more details about the Euler-Elastica and geodesics in $\Eng$. In addition, see \cite[sub-sub-Chapter 7.8.3]{agrachev} or \cite[Chapter 14]{jurdjevic} for the relation of Euler-Elastica and some \sR geodesics as the rolling problem and the Euclidean group.

\begin{figure}%
    \centering
    {{\includegraphics[width=2.5cm]{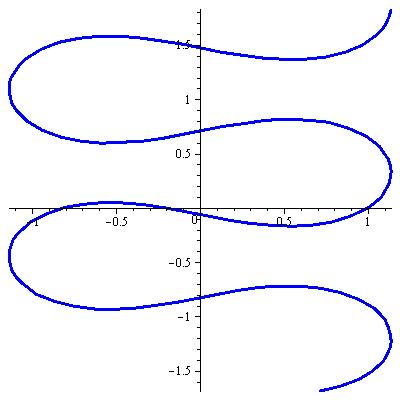} }}
    \qquad
    {{\includegraphics[width=2.5cm]{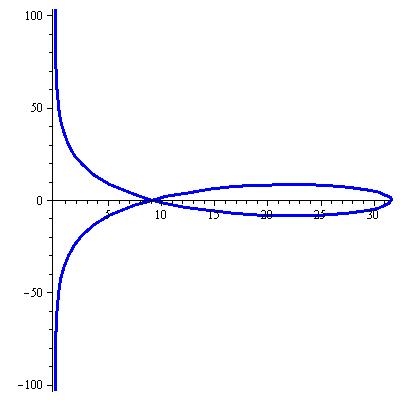} }}%
    \qquad
    {{\includegraphics[width=2.5cm]{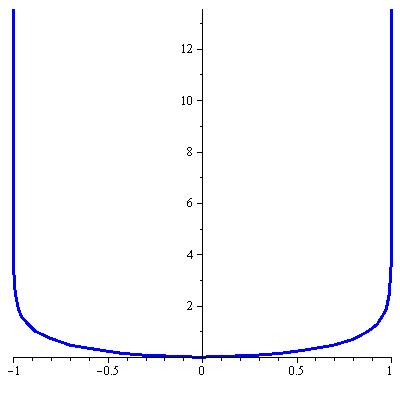} }}
    \qquad
    {{\includegraphics[width=2.5cm]{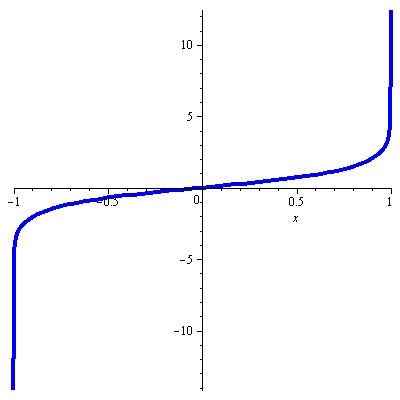} }}
    \caption{The images show the projection to $\R^2 \simeq \Je/[\Je,\Je]$, with coordinates $(x,\theta_0)$, of geodesics in $\Je$. The first panel presents a generic $x$-periodic geodesic, the second panel displays the projection of a homoclinic geodesic, which is the Euler-soliton solution to the Euler-Elastica problem and correspond to case $n =1$ from Theorem \ref{th:main-2}. The third panel presents the projection of a turn-back geodesic, the forth panel displays the projection of a heteroclinic of the direct-type geodesic.}    
    \label{fig:perio-hom-curve}
\end{figure}

\subsection*{Our Method To Prove A Geodesic Is A Metric Line}

The optimal synthesis and the weak KAM theory are the two classical methods to prove that a geodesic is a metric line. See \cite[sub-Chapter 9.4]{jurdjevic}, \cite[sub-Chapter 13.4]{agrachev} or \cite[Chapter 13]{agrachev2004control} for an introduction to the optimal synthesis. See \cite{fathi2008weak} for more details of weak KAM theory on the \Ri context, and see \cite[sub-sub-chapter 1.9.2]{tour} or \cite[Section 5]{RM-ABD} on the \sR context. The optimal synthesis requires the explicit integration of the geodesic equations, and the weak KAM theory needs a global Calibration function. In both cases, the integrability of the flows is a necessary condition. Although the \sR geodesic flow in $\Je$ is integrable, see \cite[Theorem 1.1]{ABD}, these methods cannot prove the Conjecture \ref{con:1}. On one side, the explicit integration of the equation of motion is impossible in the general case. On the second side, the local Calibration functions, found in \cite[Section 5]{RM-ABD} or \cite[sub-Section 3.2]{ABD-No-perio}, do not have a global extension.

Besides Theorem \ref{th:main-1} and \ref{th:main-2}, the main contribution of this work is the formalization of the method used in \cite{RM-ABD}. We will consider a \sR manifold $\R^3_F$, called the magnetic space, and a \sR submersion $\pi_F: \Je \to \pi_F$. Thanks to the fact that lift of metric line is a metric line (Proposition \ref{prp:sR-sub-metric-line}), it is enough to prove that if $\gamma(t)$ is a \sR geodesic corresponding to a polynomial $F$ and satisfying  the conditions of Theorems \ref{th:main-1} or \ref{th:main-2}, then the projection $\pi_F(\gamma(t))$ is a metric line in $\R^3_F$. In other words, we reduce the problem of studying metric lines in $\Je$ to studying metric lines in the magnetic space $\R^3_F$. Theorems \ref{th1:d-t-me-lin} and \ref{th1:h-me-lin} show that the curve $c(t) := \pi_F(\gamma(t))$ is a metric line, where $\gamma(t)$ is \sR geodesic given by Theorems \ref{th:main-1} and \ref{th:main-2}, respectively. To prove Theorems \ref{th1:d-t-me-lin} and \ref{th1:h-me-lin}, we consider a sequence of minimizing \sR geodesics $c_n(t)$ joining every time farther away points on the geodesic $c(t)$, see Figure \ref{fig:h}. We show that the sequence has a convergent subsequence $c_{n_j}(t)$ converging to a minimizing geodesic $c_{\infty}(t)$ corresponding to the polynomial $F$, since every two \sR geodesics corresponding to the polynomial $F$ are related by \sR isometry we conclude that $c(t)$ is a metric line.

\subsection*{Outline} 

Section \ref{sec:preliminary} introduces the preliminary result necessary to prove Theorem \ref{th:main-1} and \ref{th:main-2}. Sub-Section \ref{sub-sec:Je-as-sR} briefly describes $\Je$ as a \sR manifold and summarizes some previous results from \cite{RM-ABD}. Between them, the most important are: the  \textbf{Background Theorem} establishing the correspondence between \sR geodesics and the pair $(F,I)$, the classification of \sR geodesic, the formal definition of a \sR submersion and Proposition \ref{prp:sR-sub-metric-line}. Sub-Section \ref{sec:mag-spa} presents the magnetic space $\R^3_F$ and some previous results from \cite{RM-ABD}: The correspondence between \sR geodesic in the magnetic space $\R^3_F$ and the pair $(G,I)$ where $G$ is a polynomial in a two-dimensional space $Pen_F$, the period map $\Theta(G,I)$, and an upper bound for the cut time. In addition, sub-Section \ref{sec:mag-spa}  provides the relation between the \sR geodesic in $\R^3_F$ and $\Je$,  the cost function definition, and some sequence of \sR geodesics' properties.

The main goal of Section \ref{sec:d-t-geodesics} is to prove Theorem \ref{th1:d-t-me-lin}. Sub-Section \ref{sub-sec:d-t-mag-spa} presents a particular magnetic space for each heteroclinic geodesic of the direct-type and some essential properties of this space. In particular, Corollary \ref{cor:d-t-The} says that the Period map $\Theta(G,I)$ is one-to-one when restricted to space of heteroclinic geodesics of the heteroclinic direct-type. Sub-sub-Section \ref{sub-sub:set-up-proof} sets up the proof, sub-sub-Section \ref{sub-sec:proof-theo-1} presents the proof of Theorem \ref{th1:d-t-me-lin} and sub-sub-Section \ref{sub-sec:proof-main-the-1} provides the formal proof of Theorem \ref{th:main-1}. 

The main goal of Section \ref{sec:h-geodesics} is to prove Theorem \ref{th1:h-me-lin} and has a similar structure to Section \ref{sec:d-t-geodesics}. In addition, Section \ref{sec:h-geodesics} introduces Theorem \ref{th1:h-no-me-lin}, which says that the \sR geodesic corresponding to the polynomial $F(x) = 1-2x^{2n+1}$ is not a metric line in the magnetic space $\R^3_F$.

\subsection*{Acknowledgment}
I want to acknowledge the mathematics department at the University of California Santa Cruz for allowing me to be in the Ph.D. program and helping me succeed in my degree. This work resulted from years of working, talks, and conversations in the mathematics department. In particular, I express my gratitude to my advisor, Richard Montgomery, for setting up the problem and his ideas to approach it. I thank Andrei Ardentov, Yuri Sachkov, and Felipe Monroy-Perez, for e-mail conversations throughout this work and whose work inspired me. This paper was developed with the support of the scholarship (CVU 619610) from  "Consejo de Ciencia y Tecnologia"  (CONACYT)

\section{Preliminary}\label{sec:preliminary}

Here we will introduce the necessary results to prove Theorems \ref{th:main-1} and \ref{th:main-2}. 

\subsection{$\Je$ as a \sR manifold}\label{sub-sec:Je-as-sR}
We describe here briefly the \sR structure on $\Je$. For more details about $\Je$ as a Carnot group and \sR manifolds, see \cite[sub-Chapter 10.2]{agrachev}, \cite[sub-Section 2.1]{RM-ABD}, \cite[Section 2]{ABD} or \cite[Section 3]{CarnotJets}. We see $\Je$ as $\R^{k+2}$, using $(x,\theta_0,\dots,\theta_k)$ as global coordinates, then $\Je$ is endowed with a natural rank 2 distribution $\Di \subset T\Je$ characterized by the $k$-Pfaffian equations
\begin{equation*}
0 = d\theta_i - \frac{1}{i!}x^i d\theta_0,\qquad i = 1,\dots,k.
\end{equation*}
$\Di$ is globally framed by two vector fields
\begin{equation}\label{eq:frame}
X = \frac{\partial}{\partial x}\;\; \text{and} \;\; Y = \sum_{i = 0}^k \frac{x^i}{i!} \frac{\partial}{\partial \theta_i}.
\end{equation}
We declare these two vector fields to be orthonormal to define the \sR structure on $\Je$. The \sR metric is given by restricting $ds^2 = dx^2 + d\theta_0^2$ to $\Di$. $\Je$ is endowed with a canonical projection $\pi: \Je \to \R^2 \simeq \Je / [\Je,\Je]$, which in coordinates is given by $\pi(x,\theta_0,\dots,\theta_k) = (x,\theta_0)$. 

\subsubsection{Reduced System}

As we proved in \cite[Appendix A]{RM-ABD}, a geodesic in $\Je$ is determined by the pair $(F,I)$. Let $F$ be a polynomial $F$ of degree $k$ or less, then the reduced Hamiltonian system $H_F$ is given by
\begin{equation}\label{eq:fund-eq-jet-space}
H_F = \frac{1}{2} ( p_x^2 + F^2(x)). 
\end{equation}
The condition $\frac{1}{2} = H_F$ implies that the reduced dynamics occur in the hill interval $I$ and the \sR geodesic is parametrized by arc length. Let us formalize the hill interval definition.
\begin{definition}\label{def:hill-int}
 We say that a closed interval $I$ is a hill interval associated to $F(x)$, if $|F(x)| <1$ for every $x$ in the interior of $I$ and $|F(x)| = 1$ for every $x$ in the boundary of $I$. If $I$ is of the form $[x_0,x_1]$, then we call $x_0$ and $x_1$ the endpoints of the hill interval. We say that $hill(F)$ is the hill region of $F$ if $hill(F)$ is union of all the hill intervals of $F$.  
\end{definition}
By definition, if $F(x)$ is not a constant polynomial then $I$ is compact. In contrast, the constant polynomial $F(x)$  has hill interval $I = \R$ if $|F(x)| \leq 1$ and $I$ is equal to a single point if $F(x) = \pm 1$.

Here, we prescribe the method to build a \sR geodesic: first, find a solution to the reduced system \eqref{eq:fund-eq-jet-space} with energy $\frac{1}{2} = H_F$. Second, having found the solution $(p_x(t),x(t))$, we define a curve $\gamma(t)$ in $\Je$ by the following equation
\begin{equation*}
\dot{\gamma}(t) = \dot{x}(t)X(\gamma(t)) + F(x(t))Y(\gamma(t))\;\;\;\text{where} \;\;\; \dot{x}(t) = p(t).
\end{equation*}
The \textbf{Background Theorem} establishes the correspondent between the pair $(F,I)$ and the \sR geodesics on $\Je$.
\begin{backgroundtheorem}\label{background}
The above prescription yields a geodesic in $J^k$ parameterized by
arc length.  Conversely, any arc length parameterized geodesic in $J^k$ can be achieved
by this prescription applied to some polynomial $F(x)$ of degree   $k$ or less. 
\end{backgroundtheorem}
The \textbf{Background Theorem} was proved first in \cite{Monroy1,Monroy2,Monroy3}, later we gave an alternative proof in \cite[Appendix A]{RM-ABD}.

\subsubsection{Classification Of Geodesic In Jet Space} \label{sec:clas-geo-je}
Let $\gamma(t)$ be a geodesic in $\Je$ corresponding to the pair $(F,I)$, in the case when $F(x)$ is no constant polynomial let assume $I = [x_0,x_1]$, then $\gamma(t)$ is only one of the following options:
\begin{itemize}
\item We say that a geodesic $\gamma(t)$ is a line if the projected curve $\pi(\gamma(t))$ is a line in $\R^{2}$. Line geodesics correspond to constant polynomial or trivial solutions of the reduced dynamics.
 
\item We say $\gamma(t)$ is $x$-periodic if its reduced dynamics is periodic. The reduced dynamics is periodic if and only if $x_0$ and $x_1$ are regular points of $F(x)$.

\item We say $\gamma(t)$ is homoclinic if its reduced dynamics is a homoclinic orbit. The reduced dynamics has a homoclinic orbit if and only if one of the points $x_0$ and $x_1$ is regular and the other is a critical point of $F(x)$.

\item We say $\gamma(t)$ is heteroclinic if its reduced dynamics is a heteroclinic orbit. The reduced dynamics has a heteroclinic orbit if and only if both points $x_0$ and $x_1$ are critical of $F(x)$.

\item We say a heteroclinic geodesic $\gamma(t)$ is turn-back if $F(x_0)F(x_1) = -1$.

\item We say a heteroclinic geodesic $\gamma(t)$ is direct-type if $F(x_0)F(x_1) = 1$.
\end{itemize}
See figure \ref{fig:perio-hom-curve} for a better undertanding of these names.

\subsubsection{Unitary Geodesics}

To prove Theorem \ref{th:main-1} and \ref{th:main-2}, we will introduce the concept of a unitary geodesic.
\begin{definition}\label{def:norm-geode}
We say a geodesic $\gamma(t)$ in $\Je$ corresponding to the pair $(F,I)$ is unitary if $I = [0,1]$. We say a heteroclinic of the direct-type geodesic (or homoclinic) $\gamma(t)$ is unitary, if in addition $F(x(t))  \to 1$ when $t \to \pm \infty$. 
\end{definition}
Let us denote by $Iso(\Je)$, the isometry group of $\Je$. The reflection $R_{\theta_0}(x,\theta_0,\theta_1,\dots,\theta_k) = (x,-\theta_0,\theta_1,\dots,\theta_k)$ is in $Iso(\Je)$. Then, to classify metric lines it is enough to study unitary geodesics, since if $\gamma(t)$ is a heteroclinic of the direct-type or homoclinic geodesic such that $F(x(t))  \to -1$ when $t \to \pm \infty$, then $R_{\theta_0}(\gamma(t))$ is such that $F(x(t))  \to 1$ when $t \to \pm \infty$.

\begin{corollary}\label{cor:d-t-F}
Let $\gamma(t)$ be a unitary heteroclinic of the direct-type geodesic for $F(x)$, then there exists $q(x)$  such that $F(x) = 1-x^{k_1}(1-x)^{k_2}q(x)$, where $1 < k_1$, $1 < k_2$, and $q(x)$ is  polynomial of degree $k-k_1-k_2$ such that $0 <x^{k_1}(1-x)^{k_2}q(x) < 2$ if $x$ is in $(0,1)$.
\end{corollary}

\begin{proof}
By construction, $F(x)$ is such that 
$$F(0) = F(1) = 1, \;\; F'(0) = F'(1) = 0\;\; \text{and} \;\;|F(x)| < 1\;\text{if} \;\; x \;\; \text{is in} \;\;(0,1),$$ then using the Euclidean algorithm we find the desired result. 
\end{proof}

The following Proposition tells us that every geodesic in $\Je$ is related to unitary geodesic by a Carnot dilatation and translation.
\begin{proposition}\label{prp:normal-geo}
Let $\gamma(t)$ be a geodesic in $\Je$ associated to the pair $(F,I)$ and let $h(\tilde{x}) = x_0 + u\tilde{x}$ be the affine map taking $[0,1]$ to $I = [x_0,x_1]$ with $u:=x_1-x_0$. If $\hat{F}(h(\tilde{x})) = F(x)$ and $\hat{\gamma}(t)$ is the geodesic in $\Je$ corresponding to the pair $(\hat{F},[0,1])$. Then $\gamma(t)$ is related to $\hat{\gamma}(t)$ by Carnot dilatation and translation, that is 
$$ \gamma(t) =   \delta_{u} \hat{\gamma}(\frac{t}{u})*(x_0,0\dots,0),$$
where $\delta_{u}$ is the Carnot dilatation.
\end{proposition}
For more details about the Carnot dilatation see \cite[sub-Chapter 8.2]{tour}.

Proposition \ref{prp:normal-geo} and the reflection $R_{\theta_0}$ imply that it is enough to prove Theorem \ref{th:main-1} and \ref{th:main-2} for the unitary case.

\subsubsection{\SR Submersion}
Let us formalize the definition of \sR submersion and present Proposition \ref{prp:sR-sub-metric-line}.
\begin{definition}\label{def:sR-submersion}
 Let $(M,\Di_M,g_M)$ and $(N,\Di_N,g_N)$ be two \sR manifolds and let $\phi:M \to N$ a submersion ($dim(M) \geq dim(N)$). We say that $\phi$ is a  \sR submersion if $\phi_* \Di_M = \Di_N$ and $\phi^* g_N = g_M$.
\end{definition}

We remark that the projection $\pi: \Je \to \R^2$, defined in sub-Section \ref{sub-sec:Je-as-sR}, is a \sR submersion. As a consequence, a curve $\gamma(t)$ in $\Je$ and its projection $\pi(\gamma(t))$ have the same arc length.

A classic result on metric lines is the following.
\begin{proposition}\label{prp:sR-sub-metric-line}
Let $\phi:M \to N$ be a \sR submersion and let $c(t)$ be a metric line in $N$, then the horizontal lift of $c(t)$ is a metric line in $M$.  
\end{proposition}
The proof of Proposition \ref{prp:sR-sub-metric-line} is given in \cite[p. 154]{RM-ABD}. The following corollary is an immediate result to the Proposition \ref{prp:sR-sub-metric-line}.
\begin{corollary}\label{cor:sR-sub-metric-line}
 Geodesic lines are metric lines in $\Je$. 
\end{corollary}

\subsection{The 3-Dimensional Magnetic Space}\label{sec:mag-spa}
In \cite{RM-ABD},  we introduced the $3$-dimensional \sR manifold, denoted by $\R^{3}_{F}$ and called \enquote{magnetic \sR structure} or \enquote{magnetic space}, whose geometry depends on the choice of a polynomial $F(x)$. To endow $\R^{3}_{F}$ with the \sR structure we use global coordinates $(x,y,z)$. Then, we define the two rank non-integrable distribution $\Di_{F}$ and the \sR metric on the distribution $\Di_{F}$ by the Pfaffian equation $ dz - F(x) dy = 0$ and $ds^2_{\R^{3}_F} = ( dx^2 + dy^2)|_{\Di_F}$, respectively. We provided a \sR submersion  $\pi_{F}$ factoring the \sR submersion $\pi:\Je \to \R^{2}$, that is, $\pi = pr \circ \pi_{F}$, where the target of $\pi_{F}$ is $\R^{3}_{F}$ and the target of $pr$ is $\R^{2}$. If $F(x) = \sum_{i=0}^k a_ix^i$, then the projections $\pi_{F}$ and $pr$ are given in coordinates by
\begin{equation}\label{eq:proojection-F}
\begin{split}
\pi_{F}(x,\theta) = (x,\theta_0,\sum_{\ell=0}^{m-1} a_{\ell} \theta_\ell ) = (x,y,z), \;\; \text{and} \;\; pr(x,y,z) := (x,y) .  
\end{split}
\end{equation}
It follows that $\pi_{F}$ maps  the frame $\{X,Y \}$ defined in \eqref{eq:frame} into the frame $\{ \tilde{X},\tilde{Y}\}$, that is,
\begin{equation*}
\tilde{X} := \frac{\partial}{\partial x} = (\pi_{F})_* X \;\; \text{and}\;\; \tilde{Y} := \frac{\partial}{\partial y} + F(x)\frac{\partial}{\partial z} = (\pi_{F})_* Y .
\end{equation*} 
We conclude $\Di_{F}$ is globally framed by the \on vector fields $\{ \tilde{X},\tilde{Y} \}$.

An explanation of the names \enquote{magnetic \sR structure} or \enquote{magnetic space} is given in \cite[sub-Section 4.1]{RM-ABD}.

\subsubsection{Geodesics In The Magnetic Space}\
The Hamiltonian function governing the \sR geodesic flow in $\R^{3}_{F}$ is 
\begin{equation}\label{eq:fund-eq-F}
H_{\pi^3_F}(p_x,p_y,p_z,x,y,z) = \frac{1}{2} \sum_{i=1}^n p_{x_i}^2 + \frac{1}{2}(p_y + F(x) p_z)^2. 
\end{equation}
We say a curve $c(t) = (x(t),y(t),z(t))$  is a $\R^{3}_{F}$-geodesic parametrized by arc length in $\R^{3}_{F}$, if it is the projection of the \sR geodesic flow with the condition $H_{\pi^3_F} = \frac{1}{2}$. Since  $H_{\pi^3_F}$ does not depend on the coordinates $y$ and $z$, they are cycle coordinates, so the momentum $p_y$ and $p_z$ are constant of motion, see \cite[p. 162]{Landau} or \cite[p. 67]{Arnold} for the definition of cycle coordinates and their properties.  Since $y$ and $z$ are cycle coordinates, then the translation $\varphi_{(y_0,z_0)}(x,y,z) = (x,y+y_0,z+z_0)$ is an isometry.
\begin{definition}\label{def:s-r-distance-isome}
We denote by $dist_{\R^{3}_{F}}(\;,\;)$ and $Iso(\R^{3}_F)$, the \sR distance and the isometry group in $\R^{3}_{F}$. 
\end{definition}
For more details about the definition of \sR distance the \sR group of isometries see \cite[Chapter 1.4]{tour} or \cite[sub-Chapter 3.2]{agrachev}. Then the translation $\varphi_{(y_0,z_0)}$ is in $Iso(\R^{3}_F)$. 

Setting $p_y = a$ and $p_z = b$ inspired the following definition.
\begin{definition}\label{def:pencil}
We say that the two-dimensional linear space $Pen_{F}$ is the pencil of $F(x)$, if $Pen_{F} := \{ G(x) = a + bF(x):  (a,b) \in \R^2 \}$. 
\end{definition}
We define the lift of a curve in $\R^{3}_{F}$ to a curve in $\Je$.
\begin{definition}\label{def:mag-lift}
Let $c(t)$ be a curve in $\R^{3}_{F}$. We say that a curve $\gamma(t)$ in $\Je$ is the lift of $c(t) = (x(t),y(t),z(t))$ if $\gamma(t)$ solves
\begin{equation*}
\dot{\gamma}(t) =  \dot{x}(t) X(\gamma(t)) + G(x(t)) Y(\gamma(t)).
\end{equation*} 
\end{definition}
Now we describe the $\R^{3}_{F}$-geodesics, their lifts, and their relation with the \sR geodesics in $\Je$. 
\begin{proposition}
Let $c(t)$ be a $\R^{3}_{F}$-geodesic for $G(x)$ in $Pen_{F}$, then the component $x(t)$ satisfies the $1$-degree of freedom Hamiltonian equation
\begin{equation*}
H_{(a,b)}(p_x,x) := \frac{1}{2} p_{x}^2  + \frac{1}{2}(a+bF(x))^2 = \frac{1}{2}  p_{x}^2 + \frac{1}{2}G^2(x).
\end{equation*}
Having found a solution $(p_{x}(t),x(t))$, the coordinates $y(t)$ and $z(t)$ satisfy
\begin{equation}\label{eq:ode-y-z}
\dot{y} = G(x(t)) \;\;\text{and}\;\; \dot{z} = G(x(t)) F(x(t)).
\end{equation}
Moreover, every $\R^{3}_{F}$-geodesic is the $\pi_{F}$-projection of a geodesic in $\Je$ corresponding to $G(x)$ in $Pen_{F}$. Conversely, the lifts of  a $\R^{3}_{F}$-geodesic are precisely those geodesics corresponding to polynomials in $Pen_F$.
\end{proposition}
The proof was presented in \cite[sub-Section 4.1]{RM-ABD}.

The \sR geometry has two type of geodesics: normal and abnormal. The \sR geodesic flow defines the normal geodesics, while, the endpoint map determines the abnormal geodesics, the following Lemma characterizes the abnormal geodesics in  $\R^{3}_{F}$.
\begin{lemma}\label{lemma:abn-geo}
A curve $c(t)$ in $\R^{3}_{F}$ is an abnormal geodesic  if and only if $c(t)$ is  tangent to the vector field $\tilde{Y}$ and $x(t) = x^*$ is a constant point in $\R$ such that $F'(x^*) = 0$.
\end{lemma}
For more details about the endpoint map and abnormal geodesics, see \cite[Chapter 3]{tour}, \cite[sub-sub-Chapter 4.3.2]{agrachev} or \cite{BryantHsu}.

\begin{corollary}
Let $\gamma(t)$ be a \sR geodesic in $\Je$ corresponding to the polynomial $F(x)$ and let $c(t)$ be the curve given by $\pi_{F}(\gamma(t))$, then $c(t)$ is a $\R^{3}_{F}$-geodesic corresponding to the pencil $(a,b) = (0,1)$. 
\end{corollary}
\begin{proof}
By construction, the pencil $(a,b) = (0,1)$ correspond to the polynomial $F(x)$.
\end{proof}

We classify the \sR geodesics in $\R^3_F$ according to their reduce dynamics defined by the reduced Hamiltonian $H_{(a,b)}$, equation \eqref{eq:fund-eq-F}, in the same way as we did in sub-sub-Section \ref{sec:clas-geo-je}.

\subsubsection{Cost Map In Magnetic Space}

In \cite[sub-Section 7.2]{RM-ABD}, we defined the $Cost$ map and used it to prove the main result. Here, we introduce $Cost$, an auxiliary function to show Theorems \ref{th:main-1} and \ref{th:main-2}.
\begin{definition}\label{def:cost-f-time}
\index{$\Delta y(c,[t_0,t_1])$} \index{$\Delta z(c,[t_0,t_1])$} \index{$\Delta t(c,[t_0,t_1])$} \index{$\Delta (c,[t_0,t_1])$} Let $c(t)$ be a $\R^{3}_{F}$-geodesic defined on the interval $[t_0,t_1]$. We define the function $\Delta:(c,[t_0,t_1])$ $ \to [0,\infty] \times \R^2$ given by
\begin{equation}
\begin{split}
\Delta(c,[t_0,t_1]) & := (\Delta t (c,[t_0,t_1]),\Delta y (c,[t_0,t_1]), \Delta z (c,[t_0,t_1])) \\
                    & := (t_1 - t_0,y(t_1) - y(t_0)  ,z(t_1) - z(t_0)). \\
\end{split}
\end{equation}
And the function $Cost: (c,[t_0,t_1])$ $ \to [0,\infty] \times \R$ given by 
\begin{equation}
\begin{split}
Cost(c,[t_0,t_1])&:= ( Cost_t(c,[t_0,t_1]) , Cost_y(c,[t_0,t_1]) \\       
\end{split}
\end{equation} 
where
\begin{equation*}
\begin{split}
Cost_t(c,[t_0,t_1]) = & \Delta t(c,[t_0,t_1]) - \Delta y(c,[t_0,t_1]) \\
Cost_y(c,[t_0,t_1] = &  \Delta y(c,[t_0,t_1]) - \Delta z(c,[t_0,t_1]). \\
\end{split}
\end{equation*} 
\end{definition}
Let us prove that $Cost(c,[t_0,t_1])$ is well-defined:
\begin{proof}
By construction, 
$$|\Delta y(c,[t_0,t_1]) | \leq \Delta t(c,[t_0,t_1]),\;\;\; \text{so}\;\;\;   0 \leq Cost_t(c,[t_0,t_1]). $$
\end{proof}
We interpret $Cost_t(c,[t_0,t_1])$ as the cost that takes to the geodesic $c(t)$ travel through the $y$-component in the positive direction. To give more meaning to this interpretation, we present the following Lemma.
\begin{lemma}\label{lem:cost-fun-int}
Let $c(t)$ and $\tilde{c}(t)$ be two $\R^{3}_{F}$-geodesics. Let us assume  that they travel from a point $A$ to a point $B$ in a time interval $[t_0,t_1]$ and $[\tilde{t}_0,\tilde{t}_1]$, respectively. If $Cost_t(c_1,[t_0,t_1]) < Cost_t(c_2,[\tilde{t}_0,\tilde{t}_1])$, then the arc length of $c(t)$ is shorter that the arc length of $\tilde{c}(t)$.
\end{lemma}
\begin{proof}
We need to show that $ \Delta t(c_1,[t_0,t_1]) < \Delta t(c_2,[\tilde{t}_0,\tilde{t}_1])$. Since $A = c(t_0) = \tilde{c}(\tilde{t}_0)$ and $B = c(t_1) = \tilde{c}(\tilde{t}_1)$, it follows that 
$$\Delta y(c_1,[t_0,t_1]) = \Delta y(c_2,[\tilde{t}_0,\tilde{t}_1])$$
 which implies
\begin{equation*}
\Delta t(c_1,[t_0,t_1]) - Cost_t (c_1,[t_0,t_1]) =   \Delta t(c_2,[\tilde{t}_0,\tilde{t}_1]) - Cost_t(c_2,[\tilde{t}_0,\tilde{t}_1]),
\end{equation*}
so $0 < Cost_t(c_2,[\tilde{t}_0,\tilde{t}_1]) - Cost_t (c_1,[t_0,t_1]) =   \Delta t(c_2,[\tilde{t}_0,\tilde{t}_1]) - \Delta t(c_1,[t_0,t_1])$.
\end{proof}

\subsubsection{Periods}

The following Proposition is a classical result from classical mechanics.
\begin{proposition}\label{cor:mag-geo-period}
Let $c(t)$ be a $x$-periodic $\R^3_{F}$-geodesic for the pencil $(a,b)$ with a hill interval $I$ the period is given by
\begin{equation}\label{eq:t-period}
L(G,I) := 2 \int_{I} \frac{dx}{\sqrt{1-G^2(x)}} .
\end{equation}
Moreover, the changes $\Delta y(c,[t,t+L]) = \Delta y(G,I)$ and $\Delta z(c,[t,t+L])  = \Delta y(G,I)$  are given by 
\begin{equation}\label{eq:y-z-period}
\begin{split}
\Delta y(G,I)  & := 2 \int_{I} \frac{G(x) dx}{\sqrt{1-G^2(x)}}\;\;\text{and} \;\; \Delta z(G,I)   := 2  \int_{I} \frac{G(x)F(x) dx}{\sqrt{1-G^2(x)}}.\\
\end{split}
\end{equation}
\end{proposition}
In \cite[sub-Section 4.3]{RM-ABD}, we proved Proposition \ref{cor:mag-geo-period} using classical mechanics, see \cite[Section 11]{Landau}. In \cite[Section 2]{ABD-No-perio}, we showed a equivalent statement, in the context of $\Je$, using a generating function of the second type, see \cite[Section 50]{Arnold}.  $L(G,I)$ and $\Delta y(G,I)$  are smooth functions with respect to the parameters $(a,b)$ if and only if the corresponding geodesic $c(t)$ for $(G,I)$ is  $x$-periodic. We define an auxiliary map that will help us to prove Theorems \ref{th:main-1} and \ref{th:main-2}. 
\begin{definition}\label{def:y-z-period}
\index{$\Theta(G,I) $} The period map $\Theta:(G,I) \to [0,\infty]  \times \R$ is given by
\begin{equation*}
\begin{split}
\Theta(G,I)& := (\Theta_1(G,I),\Theta_2(G,I)) \\
           &:=  2 ( \int_{I} \sqrt{\frac{1-G(x)}{1+G(x)}} dx,  \int_{I} G(x)\frac{1-F(x)}{\sqrt{1-G^2(x)}} dx) .\\
\end{split}
\end{equation*}
\end{definition}
$\Theta_1(G,I)$ is a smooth function with respect the parameters $(a,b)$ not only when the corresponding geodesic $c(t)$ for $(G,I)$ is  $x$-periodic, $\Theta_1(G,I)$ is also smooth when $c(t)$ is a  heteroclinic of the direct-type or homoclinic geodesic such that $G(x(t))  \to 1$ when $t \to \pm \infty$.
\begin{corollary}
Let $G(x)$ be in $Pen_F$. Then:

(1) $\Theta_1(G,I) = 0$ if and only if $G(x) = 1$.

(2) If $I = [x_0,x_1]$ is compact, then  $\Theta_1(G,I)$ is finite if and only if $x_0$ and $x_1$ are not critical point of $G(x)$ with value $-1$. 
\end{corollary}

We introduce an important concept called the travel interval.
\begin{definition}\label{def:travel-I}
Let $c(t)$ be a $\R^3_{F}$-geodesic traveling during  the time interval $[t_0,t_1]$. We say that    $\mathcal{I}[t_0,t_1] := x([t_0,t_1])$ is the travel interval of  $c(t)$, counting multiplicity.    
\end{definition}
For instance, if $c(t)$ is a $\R^3_{F}$-geodesic with hill interval $I$ such that its coordinate $x(t)$ is $L$-periodic, then $\mathcal{I}[t,t+L] = 2I$.

\begin{corollary}\label{prop:mag-geo-Delta}
Let $c(t)$ be a $\R^{3}_{F}$-geodesic for $G(x)$ in $Pen_F$ with travel interval $\mathcal{I}$. Then  $\Delta (c,[t_0,t_1])$ from Definition \ref{def:cost-f-time} can be rewritten in terms of polynomial $G(x)$ and the travel interval $\mathcal{I}$ as follows;
\begin{equation*}
\begin{split}
\Delta (c,[t_0,t_1]) & =  \Delta  (G,\mathcal{I}) \\
&:=  (\int_{\mathcal{I}} \frac{dx}{\sqrt{1-G^2(x)}},\int_{\mathcal{I}} \frac{G(x)dx}{\sqrt{1-G^2(x)}},\int_{\mathcal{I}} \frac{ G(x)F(x)dx}{\sqrt{1-G^2(x)}}) .\\
\end{split}
\end{equation*}
In the same way, the map $Cost(c,[t_0,t_1])$ from Definition \ref{def:cost-f-time} can be rewritten as follows:
\begin{equation*}
\begin{split}
Cost(c,[t_0,t_1]) = Cost(G,\mathcal{I}) :=  (\int_{\mathcal{I}} \frac{1-G(x)}{\sqrt{1-G^2(x)}}dx,  \int_{\mathcal{I}} \frac{(1-F(x))G(x)}{\sqrt{1-G^2(x)}}dx)  \\        
\end{split} 
\end{equation*}
\end{corollary}
The proof is the same as that for Proposition \ref{cor:mag-geo-period}. We remark that $Cost(G,\mathcal{I})$ and $\Theta(G,I)$ are equal if and only if $I = \mathcal{I}$.

\begin{corollary} Let $c(t)$ be a $\R^3_{F}$-geodesic, then
$$\lim_{n \to \infty } Cost_{t}(c,[-n,n])\;\; \text{ is finite if and only if} \;\; \lim_{t\to \pm\infty} G(x(t)) = 1.$$ 
\end{corollary}

\subsubsection{Upper Bound of the Cut Time}
Here we introduce the cut time definition in the context of $\Je$.
\begin{definition}
Let $\gamma: \R \to \Je$ be a \sR geodesic parameterized by arc length.    The  cut time of  $\gamma$ is
$$ t_{cut} (\gamma)  := \sup \{ t>0 : \; \gamma|_{[0,t]} \;\; \text{is length-minimizing}  \}.$$
\end{definition}


\begin{proposition}\label{Maxwell}
Let $c(t)$ be a $x$-periodic geodesic on  $R^3 _F$ for the pair $(G,I)$ and with $x$-period $L$. 
Then 

1.- $t_{cut}(c) \le  \frac{L(G,I)}{2}$ if  $F$ is even and $c$'s Hill interval   contains $0$.  

2.- $t_{cut}(c) \le L(G,I)$ in all cases. 

\end{proposition}

The prove of this Proposition is in \cite[sub-Section 6.2]{RM-ABD}.

\subsubsection{Sequence Of Geodesics On The Magnetic Space }\label{sub-sec:sequence-R3-F}

Let us present two classical results on metric spaces.
\begin{lemma}\label{lem:seq-mim-geo}
Let $c_n(t)$ be a sequence of minimizing geodesics on the compact interval $\mathcal{T}$ converging uniformly to a geodesic $c(t)$, then $c(t)$ is minimizing in the interval $\mathcal{T}$.   
\end{lemma}
\begin{proof}
Let $[t_0,t_1] \subset \mathcal{T}$, since $c_n(t)$ is sequence of minimizing geodesic then $dist_{\R^{3}_{F}}(c_{n}(t_0),c_n(t_1)) = |t_1 - t_0|$ for all $n$. By the uniformly convergence, if $n \to \infty$ then $dist_{\R^{3}_{F}}(c(t_0),c(t_1)) = |t_1 - t_0|$. 
\end{proof}

\begin{proposition}\label{prp:seque-comp}
\index{$\mathcal{T}$}  \index{$K$} Let $K$ be a compact subset of $\R^{3}_{F}$ and let  $\mathcal{T}$ be a compact time interval. Let us define the following set of $\R^{3}_{F}$-geodesics
\begin{equation*}
\begin{split}  Min(K,\mathcal{T}) := \big\{ & \R^{3}_{F}\text{-geodesics}\; c(t) :   c(\mathcal{T}) \subset K  \text{ and}\;c(t) \text{ is minimizing in} \; \mathcal{T} \big\}. 
\end{split}
\end{equation*}
\index{$Min(K,\mathcal{T})$} Then  $Min(K,\mathcal{T})$ is a sequentially compact set with respect to the uniform
 topology. 
\end{proposition}

\begin{proof}
Let $c_n(t)$ be an arbitrary sequence in $Min(K,\mathcal{T})$, we must prove $c_n(t)$ has a uniformly convergent subsequence converging to $c(t)$ in $Min(K,\mathcal{T})$.
The space of geodesics $Min(K,\mathcal{T})$ is uniformly bounded and smooth in compact interval $\mathcal{T}$, then $Min(K,\mathcal{T})$ is a equi-continuous family of geodesics. By Arzela-Ascoli theorem, every sequence $c_n(t)$ in $Min(K,\mathcal{T})$ has a convergent subsequence $c_{n_s}(t)$ converging uniformly to a smooth curve $c(t)$. By Lemma \ref{lem:seq-mim-geo} $c(t)$ is  minimizing in $\mathcal{T}$.  
\end{proof}

A useful tool for the proof of Theorem \ref{th:main-1} and \ref{th:main-2} is the following.
\begin{corollary}\label{cor:iso-metr}
Let $c_1(t)$ be a $\R^{3}_{F}$-geodesic in $Min(K,\mathcal{T})$ and let $c_2(t)$  be a $\R^{3}_{F}$-geodesic. If $\varphi(x,y,z)$ is an isometry such that $c_2(\mathcal{T}') \subset \varphi(c_1(\mathcal{T}))$, then $c_2(t)$ is in $Min(\varphi(K),\mathcal{T}')$. 
\end{corollary}

\section{Heteroclinic Of The Direct-Type Geodesic}\label{sec:d-t-geodesics}

This section is devoted to proving Theorem \ref{th:main-1}. Let $\gamma_{d}(t)$ be an arbitrary heteroclinic of the direct-type geodesic in $\Je$ for a polynomial $F_{d}(x)$. We will consider the  space $\R^3_{F_{d}}$ and the $\R^3_{F_{d}}$-geodesic $c_{d}(t) := \pi_{F_{d}}(\gamma_{d}(t))$. Then we will prove the following Theorem. 
\begin{theorem1}\label{th1:d-t-me-lin}
Let $\gamma_{d}(t)$ be an arbitrary heteroclinic of the direct-type geodesic in $\Je$ for a polynomial $F_{d}(x)$. If $c_{d}(t) := \pi_{F_{d}}(\gamma_{d}(t))$, then $c_{d}(t)$ is a metric line $\R^3_{{F_{d}}}$.  
\end{theorem1}

Without loss of generality, let us assume $\gamma_{d}(t)$ is a unitary geodesic and let $F_{d}(x)$ has the form given by Corollary \ref{cor:d-t-F}. The goal is to show that for arbitrary $T$ the geodesic is minimizing in the interval $[-T,T]$, the strategy to verify the goal is the following: For all $n >T$, we will take a sequence of geodesics $c_n(t)$ minimizing in the interval $[-n,n]$ and joining the points $c_d(-n)$ and $c_d(n)$, see \ref{fig:h}. Then, we will find convergent subsequence $c_{n_j}(t)$ converging to a $\R^3_{F_d}$-geodesic $c_{\infty}(t)$ in $Min(K,\mathcal{T})$ and isometry $\varphi$ in $Iso(\R^3_{F_{d}})$ such that $c([-T,T]) \subseteq \varphi( c_{\infty}(\mathcal{T}))$, where $K$ is a compact subset of $\R^3_{F_d}$ and $\mathcal{T}$ is a compact interval.  By corollary \ref{cor:iso-metr}, $c_d(t)$ is minimizing in  $[-T,T]$. Since $T$ is arbitrary, $c_d(t)$ is a metric line.

Let $c_{d}(t) = (x(t),y(t),z(t))$. Without loss of generality, we can assume that $0 \leq \dot{x}_{}(t)$ and  $c_{d}(0) = (x(0),0,0)$ for some $x(0)$ in $(0,1)$ since the proof for the case $0 \geq \dot{x}(t)$ is similar and we can use the $t$, $y$, and $z$ translations.

\subsection{The Magnetic Space For Heteroclinic Geodesic}\label{sub-sec:d-t-mag-spa}

\begin{corollary}\label{cor:d-t-The}
Let $q_{max}$ be equal to $\max_{x \in [0,1] } \{ x^{k_1}(1-x)^{k_2}q(x) \}$, where $q(x)$, $k_1$ and $k_2$ are given by Corollary \ref{cor:d-t-F}. The set of all the heteroclinic of the direct-type $\R^3_{F_{d}}$-geodesic with hill interval $[0,1]$ is given by 
$$Pen_{d} := Pen_{d}^+ \cup Pen_{d}^- ,$$
where
\begin{equation*}
\begin{split}
Pen_{d}^+ := & \{ (a,b) =  (s,1-s)  :  \;\; s  \in (\frac{2}{q_{max}},1) \}, \\
Pen_{d}^- := & \{ (a,b) = (-s,s-1)  : \;\; s  \in (\frac{2}{q_{max}},1)\}. \\
\end{split}
\end{equation*}
Moreover, the map $\Theta_2 (G,[0,1]) : Pen_{d}^+ \to \R$ is one to one, and the cost map $Cost(c_{d},[t_0,t_1])$ is bounded by $\Theta_{d} := \Theta_1(F_d,[0,1])$ for all $[t_0,t_1]$.   
\end{corollary}

\begin{proof}
Since $F_{d}(x) \neq -1$ if $x$ is in $[0,1]$, the constant $\Theta_{d}$ is finite. Let us prove that $Cost(c_{d},[t_0,t_1])$ is bounded by $\Theta_{d}$ for all $[t_0,t_1]$.  Using Corollary \eqref{prop:mag-geo-Delta} and the condition $|F_{d}(x)| \leq 1$ for $x$ in $[0,1]$, we find that:
\begin{equation*}
\begin{split}
 |Cost_y(c_{d},[t_0,t_1])| < & Cost_t(c_{d},[t_0,t_1]) \\
 &   <  2\int_{[0,1]}  \sqrt{\frac{1-F_{d}(x)}{1+F_{d}(x)}} dx =: \Theta_1(F_d,[0,1]). \\
\end{split}
\end{equation*}
   
To prove that $ \Theta_2 (G,[0,1]) : Pen_{F_{d}}^+ \to \R$ is one to one,  we consider the one-parameter family of polynomials $G_s(x) = s+(1-s)F_{d}(x) $. Thus, $ \Theta_2 (G_s,[0,1]) : (0,q_{max}) \to \R$ is one variable function, let us calculate its derivative:
\begin{equation*}
\begin{split}
\frac{d}{ds}  \Theta_2 (G_s,[0,1]) & =   \frac{d}{ds} \int_{[0,1]} \frac{(1-F_{d}(x))G_s(x)}{\sqrt{1-G_s^2(x)}}dx = \int_{[0,1]} \frac{1-F_{d}(x)}{(1-G_s^2(x))^{\frac{3}{2}}}dx. \\
\end{split}
\end{equation*}
Since $ 0 < 1-F_{d}(x)$, then $0 <\frac{d}{ds} \Theta_2 (G_s,[0,1])$.
\end{proof}

We remark that $Pen_{d}^+$ defines the heteroclinic geodesics of the direct-type such that $\lim_{t \to \pm \infty}y(t) = \infty$, while $Pen_{d}^-$ defines the heteroclinic geodesics of the direct-type such that $\lim_{t \to \pm \infty}y(t) = -\infty$.

\begin{lemma}\label{lem:d-t-cal}
Let $\Omega(F_d) = hill(F_d) \times \R^2$ be the region, then  $c_{d}(t)$ is minimizing between the curves that lay in  the region $\Omega(F_d)$.
\end{lemma}
The proof is consequence of the calibration function defined on the region $\Omega(F_d)$ and  provided in \cite[Section 5]{RM-ABD}.

\begin{corollary}\label{cor:d-t-T-start}
There exist  $T^*_{d} > 0$ such that $y_{d}(t) > 0$ if $T^*_{d} < t$, and $y_{d}(t) < 0$ if $-T^*_{d} > t$.
\end{corollary}
\begin{proof}
By construction, $\lim_{t\to \infty}  y_{d}(t) = \infty$ and $\lim_{t\to -\infty} \Delta y_{d}(0) = -\infty$.
\end{proof}

\begin{definition}\label{def:comp-k}
If $\mathcal{T} :=  [t_0,t_1]$, we define the following set
\begin{equation}
\begin{split}
Com([0,1]) := \big\{  (c(t),\mathcal{T}): & \;\; c(t)\;\text{is a}\; \R^3_{F_{d}}\text{-geodesic},\\
 &\;   x(t_0) \in [0,1]  \text{ and}\;x(t_1) \in [0,1] \big\}.
\end{split}
\end{equation} 
\end{definition}

\begin{lemma}\label{lem:d-t-bounded-trav-interval}
Let us consider a sequence $(c_n(t),[-n,n])$ in $Com([0,1])$. If $Cost(c_n,[-n,n])$ is uniformly bounded, then there exists a compact subset $K_{x}$ of $\R$ such that $\mathcal{I}_n[-n,n] \subset K_{x}$ for all $n$.
\end{lemma}

The proof is Appendix \ref{proof:bounded-trav-interval}.

\subsection{ Proof of Theorem \ref{th1:d-t-me-lin}}\label{sub-sec:proof-the-A}

\subsubsection{ Set Up The Proof Of Theorem \ref{th1:d-t-me-lin}}\label{sub-sub:set-up-proof}

Let $T$ be arbitrarily large and consider the sequence of points $c_{d}(-n)$ and $c_{d}(n)$ where $T <n$ and $n$ is in $\mathbb{N}$. Let $c_n(t) = (x_n(t),y_n(t),z_n(t))$ be a sequence of minimizing $\R^3_{F_{d}}$-geodesics, in the interval $[0,T_n]$  such that:
\begin{equation}\label{eq:d-t-end-mat-con}
c_n(0) = c_{d}(-n),\;\;  c_n(T_n) = c_{d}(n) \;\; \text{and} \;\; T_n \leq n.
\end{equation}
We call the equations and inequality from \eqref{eq:d-t-end-mat-con} the endpoint conditions and the shorter condition, respectively. If $c_n(t)$ is geodesic for the polynomial $G_n(x)$ and a hill interval $I_n$, then Proposition \ref{Maxwell} implies $T_n \leq L(G_n,I_n)$. 
Since the endpoint condition holds for all $n$, then the sequence $c_{n}(t)$ holds asymptotic conditions:
\begin{equation}\label{eq:d-t-asy-cond}
\begin{split}
\lim_{n\to\infty} c_{n}(0) & = (0,-\infty,-\infty), \;\; \lim_{n\to\infty} c_{n}(T_n) = (1,\infty,\infty), \\
\end{split}
\end{equation} 
and the asymptotic period condition:
\begin{equation}\label{eq:d-t-period-asy-cond}
\begin{split}
\lim_{n \to \infty} Cost_y(c_n,[0,T_n] ) & = \frac{1}{2}\Theta_2(F_d,[0,1]) . \\
\end{split}
\end{equation}

\begin{figure}%
    \centering
    {{\includegraphics[width=7cm]{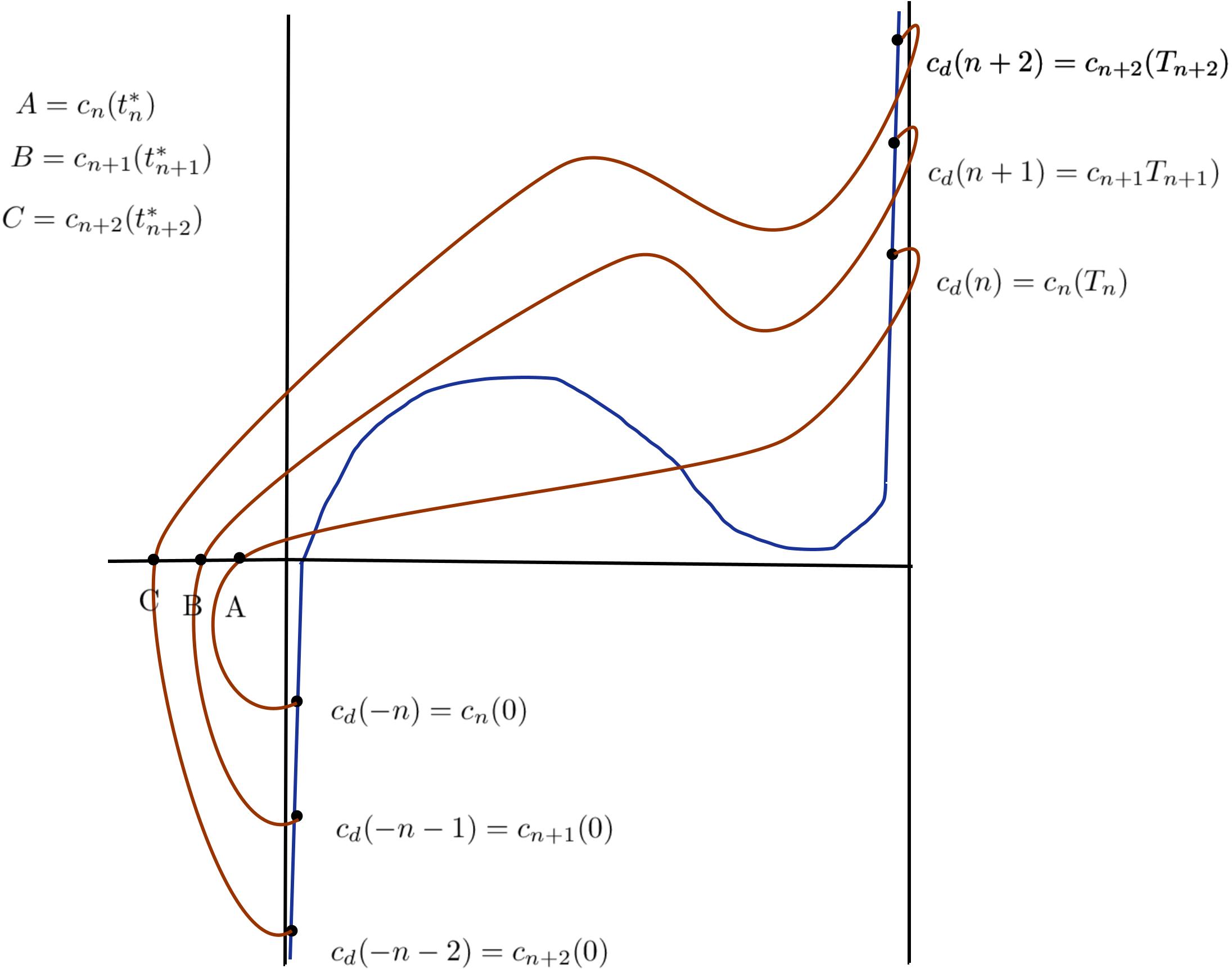} }}
    \caption{The images show the projection to $\R^2$, with coordinates $(x,y)$, of a heteroclinic of the direct-type geodesic $c_d(t)$ and the sequence of geodesics $c_n(t)$.} 
    \label{fig:h}
\end{figure}

\begin{corollary}\label{cor:non-line-geo}
The sequence of $\R^3_{F_{d}}$-geodesics $c_n(t)$ is not a sequence of geodesic lines and does not converge to a geodesic line. In particular, $c_n(t)$ does not converge to an abnormal geodesic.
\end{corollary}

\begin{proof}
Lemma \ref{lem:d-t-cal} implies that if $c_n(t)$ is shorter than $c_d(t)$, then $c_n(t)$ must leave the region $\Omega(F_d)$ and come back. Thus $c_n(t)$ is a geodesic for non-constant polynomial $G_n(x)$, and $c_n(t)$ is not a geodesic line.

Let $\mathcal{I}_n$ travel interval of $c_n(t)$, then $c_n(t)$ cannot converge to a geodesic line, since $\lim_{n\to \infty}\mathcal{I}_n = [0,1]$ and the only line in the plane $(x,y)$ that travel from $y = -\infty$ into $y = \infty$ in a fine travel interval is the vertical line, but the travel interval of the vertical line is a single point. In particular, Lemma \ref{lemma:abn-geo} implies $c_n(t)$ cannot converge to an abnormal geodesic. 
\end{proof}

The construction of the $\R^3_{F_{d}}$-geodesic $c_n(t)$ is such that the initial condition $c_n(0)$ is not bounded. The following Proposition provides a bounded initial condition.
\begin{proposition}\label{prp:d-r-ini-con-comp}
Let $n$ be a natural number larger than  $T^*_{d}$, where $T^*_{d}$ is given by Corollary \ref{cor:d-t-T-start}, and let $K_0 := K_{x} \times [-1,1] \times [-\Theta_{d},\Theta_{d}]$ be a compact set, where $K_{x}$  is the compact set from Lemma  \ref{lem:d-t-bounded-trav-interval} and $\Theta_{d}$ the constant provided by Corollary \ref{cor:d-t-The}. Then there exist a time $t_n^* \in (0,T_n)$ such that $c_n(t_n^*)$ is in $K_0$ for all $n > T^*_{d}$.
\end{proposition}
\begin{proof}
Let $n$ be a natural number larger than  $T^*_d$. By construction, $y_n(0) < 0$ and $y_n(T_n) > 0$, the intermediate value theorem implies that exist a $t_n^*$ in $(0,T_n)$ such that $y_n(t_n^*) = 0$. 
Since $Cost(c_{n},[0,T_n])$ is bounded, by Lemma \ref{lem:d-t-bounded-trav-interval}, there exists a compact set $K_{x}$ such that $x_n(t)$ is in $K_{x}$ for all $t$ in $[0,T_n]$. 

Let us prove that $|z_n(t_n^*)| \leq \Theta_{d}$: the endpoint conditions imply 
$$\Delta y(c_{d},[-n,n]) = \Delta y(c_{n},[0,T_n])\;\; \text{and} \;\; \Delta z(c_{d},[-n,n]) = \Delta z(c_{n},[0,T_n]).$$
So $Cost_y (c_{d},[-n,n]) = Cost_y (c_{n},[0,T_n])$ and by definition of  $Cost_y$, it follows that:
\begin{equation*}
\begin{split}
 z_n(t^*_n) - z_n(0) = \Delta z(c_n,[0,t^*_n]) & =  \Delta y(c_n,[0,t^*_n]) - Cost_y(c_n,[0,t^*_n]),  \\
z_{d}(0) - z_{d}(-n) =  \Delta z(c_{d},[-n,0]) & = \Delta y(c_{d},[-n,0]) - Cost_y(c_{d},[-n,0]). \\
\end{split}
\end{equation*}
By construction, $\Delta y(c_n,[0,t^*_n])  = \Delta y(c_{d},[-n,0])$, $z_{d}(0) = 0$ and $z_n(0) = z_{d}(-n)$, then 
$$| z_n(t^*_n) |  = |Cost_y(c_n,[0,t^*_n]) - Cost_y(c_{d},[-n,0])| \leq   \Theta_{d}, $$
since Corollary \ref{cor:d-t-The} says $Cost_y (c_{n},[0,T_n])$ is bounded by $\Theta_{d}$. We just proved $c_n(t_n^*)$ is in $K_0$.
\end{proof}

Let us reparametrize the sequence of minimizing $\R^3_{F_{d}}$-geodesics $c_{n}(t)$. Let $\tilde{c}_{n}(t)$ be a minimizing $\R^3_{F_{d}}$-geodesic in the interval $\mathcal{T}_n:=[-t^*_n,T_n-t^*_n]$ given by $\tilde{c}_{n}(t) := c_{n}(t + t^*_n)$. Then, $\tilde{c}_{n}(0)$ is bounded and $\tilde{c}_{n}(t)$ is a minimizing $\R^3_{F_{d}}$-geodesics in the interval $\mathcal{T}_n$. 
\begin{corollary}\label{cor:d-t-sub-seq-inter}
There exists a subsequence $\mathcal{T}_{n_j}$ such that $\mathcal{T}_{n_j} \subset \mathcal{T}_{n_{j+1}}$.
\end{corollary}
\begin{proof}
On one side $\tilde{c}_{n}(0)$ is bounded, on the other side $c(-t^*_n)$ and $c(T_n-t^*_n)$ are unbounded. Then  $[-t^*_n,T_n-t^*_n] \to [-\infty,\infty]$ when $n \to \infty$, and we can take a subsequence of intervals $\mathcal{T}_{n_j}$ such that $\mathcal{T}_{n_j} \subset \mathcal{T}_{n_{j+1}}$. 
\end{proof}
For simplicity, we will use the notation $\mathcal{T}_{n}$ for the subsequence $\mathcal{T}_{n_j}$.

\begin{lemma}\label{prp:d-t-convergent-sequece}
Let $N$ be a natural number larger than $T^*_d$. Then there exist compact set $K_N \subset \R^{3}_F$ such that $c_n(t)$ is in $Min(K_N,\mathcal{T}_N)$ if $n > N$.
\end{lemma}
\begin{proof}
Since $\tilde{c}_{n}(t)$  is minimizing on the interval $\mathcal{T}_{n}$, it follows that $\tilde{c}_{n}(t)$  is minimizing on the interval $\mathcal{T}_N \subset \mathcal{T}_{n}$ if $n >N$. Moreover, there exists a compact set $K_N$ such that $\tilde{c}_{n}(\mathcal{T}_N) \subset K_N$, since $c_n(0)$ is in $K_0$ and $c_n(t)$ is a family of smooth functions defined on the compact set $\mathcal{T}_N$.
\end{proof}
Therefore, $\tilde{c}_{n}(t)$ has a convergent subsequence $\tilde{c}_{n_j}(t)$ converging to a $\R^3_{F_{d}}$-geodesic $c_{\infty}(t)$ in $Min(K_n,\mathcal{T}_N)$. Corollary \ref{cor:non-line-geo} implies that $c_{\infty}(t)$ is a normal $\R^3_{F_{d}}$-geodesic, then we can associate $c_{\infty}(t)$ to a polynomial $G(x)$ in $Pen_{F_d}$.
The following Lemma tells the $G(x) = F_{d}(x)$.
\begin{lemma}\label{prp:d-t-F-uniqueness}
$G(x) = F_{d}(x)$ is the unique polynomial in the pencil of $F_{d}(x)$ satisfying the asymptotic conditions given by \eqref{eq:d-t-asy-cond} and \eqref{eq:d-t-period-asy-cond}.
\end{lemma}

\begin{proof}
By Proposition \ref{prp:seque-comp}, $\tilde{c}_{n}(t)$ has a convergent subsequence $\tilde{c}_{n_s}(t)$ converging to a minimizing geodesic $\tilde{c}(t)$ on the interval $\mathcal{T}_N$. Being a $\R^3_{F_{d}}$-geodesic, $c(t)$ is associated to a polynomial $G(x) = a +bF_{d}(x)$. $G(0)= a+b$ must be equal $1$, to satisfy the asymptotic conditions given by \eqref{eq:d-t-asy-cond}. Then $(a,b)$ is in $Pen_{d}^+$, the set defined in Corollary \ref{cor:d-t-The}. Since the map $\Theta_1(a,b):Pen_{d}^+ \to \R$ is one to one, the unique polynomial in $Pen_{d}^+$ satisfying the condition  \eqref{eq:d-t-asy-cond} and \eqref{eq:d-t-period-asy-cond} is $G(x) = F_{d}(x)$. 
\end{proof}

\subsubsection{ Proof of Theorem \ref{th1:d-t-me-lin}}\label{sub-sec:proof-theo-1}

\begin{proof}
Let $\tilde{c}_n(t)$ be the sequence of geodesics defined by the endpoint conditions \eqref{eq:d-t-end-mat-con}. By Lemma \ref{prp:d-t-convergent-sequece}, for all $N > T^*_d$ there exist a compact set $K_N$ such that $\tilde{c}(t)$ is in $Min(K_N,\mathcal{T}_N)$ if $n > N$. By Proposition \ref{prp:seque-comp}, there exist a subsequence $\tilde{c}_{n_j}(t)$ converging to a $\R^3_{F_{d}}$-geodesic $c_{\infty}(t)$ in $Min(K_N,\mathcal{T}_N)$.
Corollary \ref{cor:non-line-geo} implies that $c_{\infty}(t)$ is a normal geodesic for a polynomial $G(x)$ in $Pen_{F_d}^+$. Lemma \ref{prp:d-t-convergent-sequece} tells that $G(x) = F_d(x)$.

Since $c_{\infty}(t)$ and $c_{d}(t)$ are $\R^3_{F_{d}}$-geodesics for $F_{d}(x)$ with the same hill interval, there exists a translation $\varphi_{(y_0,z_0)}$ in $Iso(\R^3_{F_{d}})$ sending $c_{\infty}(t)$ to $c_{d}(t)$. Using $N$ is arbitrary and $c_{d}([-T,T])$ is bounded, we can find compact sets
$K := K_N$ and $\mathcal{T} := \mathcal{T}_N$ such that $c_{d}([-T,T]) \subset \varphi_{(y_0,z_0)} (c_{\infty}(\mathcal{T}))$ and $c_{\infty}$ is in $Min(K,\mathcal{T})$. Corollary \ref{cor:iso-metr} implies that $c_{d}(t)$ is minimizing in $[-T,T]$ and $T$ is arbitrarily. Therefore,  $c_{d}(t)$ is a metric line in $\R^3_{F_d}$.
\end{proof}

\subsubsection{ Proof of Theorem \ref{th:main-1}}\label{sub-sec:proof-main-the-1}

\begin{proof}
Let $\gamma_{d}(t)$ be an arbitrary heteroclinic of the direct-type geodesic. By Theorem \ref{th1:d-t-me-lin}, $c_{d}(t) := \pi_{F_{d}}(\gamma_{d}(t))$ is a metric line in $\R^3_{F_d}$. Since $\pi_{F_d}$ is a \sR submersion and $\gamma_d(t)$ is the lift of $c_d(t)$, then Proposition \ref{prp:sR-sub-metric-line} implies $\gamma_{d}(t)$  is a metric line in $\Je$.
\end{proof}

\section{Homoclinic Geodesics In Jet Space} \label{sec:h-geodesics}

This chapter is devoted to proving Theorem \ref{th:main-2}. Let $\gamma_{h}(t)$ be the homoclinic geodesic in $\Je$ for $F_{h}(x) := \pm(1-bx^{2n})$.  We will consider the space $\R^3_{F_{h}}$ and the geodesic $c_{h}(t) := \pi_{F_{h}}(\gamma_{h}(t))$, then we will prove the following Theorem.
\begin{theorem1}\label{th1:h-me-lin}
Let $\gamma_{h}(t)$ be an arbitrary homoclinic geodesic in $\Je$ for the polynomial $F_{h}(x) := \pm(1-bx^{2n})$. If $c_{h}(t) := \pi_{F_{h}}(\gamma_{h}(t))$, then $c_{h}(t)$ is a metric line $\R^3_{{F_{h}}}$.  
\end{theorem1}
Without loss of generality, we will consider the polynomial $F_{h}(x) := 1-2x^{2n}$ with hill intervl $[0,1]$. The strategy to prove Theorem \ref{th1:d-t-me-lin} is the same as the one used for Theorem \ref{th1:h-me-lin}.

Before prove Theorem \ref{th1:h-me-lin}, we present the following. 
\begin{theorem1}\label{th1:h-no-me-lin}
Let $\R^3_{F_{h}}$ be the magnetic space for the  polynomial $F_h(x) := 1-2x^{2n+1}$ and let $c(t)$ be the homoclinic $\R^3_{F_{h}}$-geodesic corresponding to $F_h(x)$. Then $c(t)$ is not a metric line $\R^3_{F_{h}}$. 
\end{theorem1}
Theorem \ref{th1:h-no-me-lin} say that we cannot use the the magnetic space $\R^3_F$ to the Conjecture \ref{con:1} for the general homoclinic case. Since the method used to prove Theorem \ref{th1:h-me-lin} does not work for the odd case $F(x) := 1-2x^{2n+1}$. The proof of Theorem \ref{th1:h-no-me-lin} is in Appendix \ref{sec:th1_:h-non-me-lin}.

\subsection{The Magnetic Space For the Homoclinic Geodesics}

Without loss of generality, $c_{h}(0) = (1,0,0)$, by use of the $t$, $y$ and $z$ translations. By the time reversibility of the reduced Hamiltonian $H_F$ given by \eqref{eq:fund-eq-jet-space}, it follows that $x(-n) = x(n)$ and $\Delta x(c_h,[-n,n])  := x(n) - x(-n) = 0$ for all $n$.

\begin{lemma}\label{lem:h-The}
Let $c_{h}(t)$ be the homoclinic $\R^3_{F_{h}}$-geodesic for $F_{h}(x) := 1-2x^{2n}$,  then
$$ \Theta_2(F_h,[0,1]) < 0. $$
\end{lemma}
\begin{proof}
By construction, $-x F'_h(x) = (2n-1)(1-F_h(x))$. Using integration by parts it follows that
\begin{equation*}
\begin{split}
\Theta_2(F_h,[0,1]) & = \frac{-2}{2n-1} \int_{[0,1]} \frac{xF'_h(x)F(x)dx}{\sqrt{1-F^2_h(x)}} \\
                              &  = \frac{2}{2n-1}  x\sqrt{1-F^2_h(x)}|_0^1 -  \frac{2}{2n-1} \int_{[0,1]} \sqrt{1-F_h^2(x)} dx.
\end{split}
\end{equation*} 
The desired result follows by $ x\sqrt{1-F^2_h(x)}|_0^1 = 0$.
\end{proof}

\begin{corollary}\label{cor:h-The}
The set of all the homoclinic $\R^3_{F_{h}}$-geodesics is given by
$$ Pen_{h}  := Pen_{h}^+ \cup Pen_{h}^- ,$$
where
\begin{equation*}
\begin{split}
Pen_{h}^+ &  := \{ (a,b) = (s,1-s)  : \;\; s  \in (1,\infty) \}  \\
Pen_{h}^- &  := \{ (a,b) = (-s,s-1)  : \;\; s  \in (1,\infty) \}.
\end{split}
\end{equation*}
Moreover, the map $\Theta_2 (G,[0,1]) : Pen_{h}^+ \to \R$ is one to one and the cost map $Cost(c_{h},[t_0,t_1])$ is bounded by $\Theta_1(F_h,[0,1]) := \Theta_h$ for all $[t_0,t_1]$.      
\end{corollary}

\begin{proof}
The proof's first part is the same as the one from \ref{cor:d-t-The}. To prove that $ \Theta_1 (a,b) : Pen_{h}^+ \to \R$ is one to one, we consider the one-parameter family of homoclinic polynomial $G_s(x) := s - (1-s)F_{h}(x)$ with hill interval $[0,\sqrt[2n]{\frac{1}{s}}]$. Thus, $\Theta_1 (G_s,[0,\sqrt[2n]{\frac{1}{s}}]) : (0,\infty) \to \R$ is a one variable function and it is enough to show it is a monotone increasing function. Let us set up the change of variable $x =  \sqrt[2n]{\frac{1}{s}} \tilde{x}$, then $F(\tilde{x}) = 1-2\tilde{x}^{2n} = F_h(\tilde{x})$ and
\begin{equation*}
\begin{split}
\Theta_2 (G_s,[0,\sqrt[2n]{\frac{1}{s}}]) & =    \int_{[0,\sqrt[2n]{\frac{1}{s}}]} \frac{2x^{2n}G_s(x)}{\sqrt{1-G_s^2(x)}}dx  = (\sqrt[2n]{\frac{1}{s}})^{n+1} \Theta_2(F_h,[0,1]). \\
\end{split}
\end{equation*}
Since $\frac{1}{s}$ is monotone decreasing and $\Theta_2(F_h,[0,1])$ is negative.
Then, we conclude $\Theta_2 (G_s,[0,\sqrt[2n]{\frac{1}{s}}])$ is a monotone increasing function with respect to $s$.
\end{proof}

We remark that $Pen_{h}^+$ defines the homoclinic geodesics such that 
$$\lim_{t \to \infty}y(t) = \infty\;\; \text{and} \;\; \lim_{t \to -\infty}y(t) = -\infty,$$
 while $Pen_{h}^-$ defines the homoclinic geodesics such that 
 $$\lim_{t \to \infty}y(t) = -\infty\;\; \text{and} \;\; \lim_{t \to -\infty}y(t) = \infty.$$

\begin{corollary}\label{cor:h-T-start}
There exist  $T^*_{h}>0$ such that $y_{h}(t) > 0$ if $T_h^* < t$ and $y_{h}(t) < 0$ if $-T_h^* > t$. Moreover, $Cost_y(c_h,[-t,t]) < 0$ if $T_h^* < t$.
\end{corollary}
\begin{proof}
Since $Cost_y(c_h,[-t,t]) \to \Theta_2(F_h,[0,1])$ as $t \to \infty$ and $\Theta_2(F_h,[0,1]) < 0$, we can find the desired  $T^*_{h}$. The rest of the proof is equal to Corollary \ref{cor:d-t-T-start}.
\end{proof}

\subsection{Set Up The Proof Of Theorem \ref{th1:h-me-lin}}\label{sub-sec:proof-theo-2}

Let $T$ be arbitrarily large and consider the sequence of points $c_{h}(-n)$ and $c_{h}(n)$ where $T <n$ and $n$ is in $\mathbb{N}$. Let $c_n(t) = (x_n(t),y_n(t),z_n(t))$ be a sequence of minimizing $\R^3_{F_{h}}$-geodesics in the interval $[0,T_n]$ such that:
\begin{equation}\label{eq:h-end-mat-con}
c_n(0) = c_{h}(-n),\;\;  c_n(T_n) = c_{h}(n) \;\; \text{and} \;\; T_n \leq n.
\end{equation}
We call the equations and inequality from \eqref{eq:h-end-mat-con} the endpoint conditions and the shorter condition, respectively. If $c_n(t)$ is geodesic for the polynomial $G_n(x)$ and a hill interval $I_n$, then Proposition \ref{Maxwell} implies $T_n \leq L(G_n,I_n)$. Since the endpoint condition holds for all $n$, the sequence $c_{n}(t)$ has the asymptotic conditions:
\begin{equation}\label{eq:h-asy-cond}
\begin{split}
\lim_{n\to\infty} c_{n}(0) & = (0,-\infty,-\infty), \;\; \lim_{n\to\infty} c_{n}(T_n) = (0,\infty,\infty), \\
\end{split}
\end{equation} 
and the asymptotic period condition
\begin{equation}\label{eq:h-period-asy-cond}
\begin{split}
\lim_{n \to \infty} Cost_y(c_n,[0,T_n] ) & = \Theta_{2}(F_h,[0,1]). \\
\end{split}
\end{equation}

The following Corollary tells us $c_n(t)$ is not a sequence of line geodesics. We remark that applying the calibration function found in \cite[Section 5]{RM-ABD} is only possible for every sub-interval of the time intervals $(-\infty,0)$ or $(0,\infty)$, in other words the calibration method does not work on an interval containing the time $t=0$, which correspond to the point when the $x$ coordinate bounce on the point $x=1$, for more details see \cite[Section 5]{RM-ABD}.
\begin{corollary}\label{cor:h-non-line-geo}
Let $n$ be larger than $T^*_h$, where $T^*_h$ is given by Corollary \ref{cor:h-T-start}, then the sequence of geodesics $c_n(t)$ neither is a sequence of geodesic lines, nor converge to a geodesics line. In particular, $c_n(t)$ does not converge to an abnormal geodesic.
\end{corollary}
\begin{proof}
Let us assume that $c_n(t)$ is a sequence of geodesic lines. Since $\Delta x(c_h,[-t,t]) ) = 0$ for all $n$ and $\Delta y(c_h,[-t,t]) ) > 0$ for all $n> T^*_h$, the unique geodesic line satisfying these conditions is the vertical line, which is generated by the polynomial $G_n(x) = 1$. Since $1 -F_h(x) > 0$ for all $x$, then $(1 -F_h(x))G_n(x) > 0$ for all $x$ and it follows that:
$$ Cost_y (c_n,[0,T_n]) = \int_{0}^T (1 -F_h(x(t)))G_n(x(t))dt  > 0. $$
This contradicts the endpoint conditions given by \eqref{eq:h-end-mat-con} since if $T_h^* < t$ then $Cost_y(c_h,[-t,t]) < 0$ . The same proof follows if $c_n(t)$ converges to a geodesics line $c(t)$ generated by $G(x) =1$, since there exists $N$ big enough that $G_n(x) > \frac{1}{2}$ for $n > N$. 
\end{proof}

Notice that this proof cannot be done in the case $F_h(x) = 1-2x^{2n+1}$. In Appendix \ref{sec:th1_:h-non-me-lin} under the hypothesis $F_h(x) = 1-2x^{2n+1}$, we will find a sequence of curves $c_n(t)$ shorter than $c_h(t)$ that converges to the abnormal geodesic.

The following Proposition provides the bounded initial condition.
\begin{proposition}
Let $n$ be a natural number larger than  $T^*_h$, where $T^*_h$ is given by Corollary \ref{cor:h-T-start}, and let $K_0 = K_{x} \times [-1,1] \times [-C_h,C_h]$ be a compact set, where $K_{x}$ is the compact set from Lemma  \ref{lem:d-t-bounded-trav-interval} and $C_{h}$ is constant provided by Corollary \ref{cor:h-The}. Then there exist a time $t_n^* \in (0,T_n)$ such that $c_n(t_n^*)$ is in $K_0$ for all $n > T^*_1$.
\end{proposition}
Same proof as Proposition \ref{prp:d-r-ini-con-comp}.

Let us reparametrize the sequence of minimizing $\R^3_{F_{h}}$-geodesics $c_{n}(t)$. Let $\tilde{c}_{n}(t)$ be a minimizing $\R^3_{F_{h}}$-geodesic in the interval $\mathcal{T}_n:=[-t^*_n,T_n-t^*_n]$ given by $\tilde{c}_{n}(t) := c_{n}(t + t^*_n)$. Then, $\tilde{c}_{n}(0)$ is bounded and  $\tilde{c}_{n}(t)$ is a minimizing $\R^3_{F_{d}}$-geodesics in the interval $\mathcal{T}_n$.  
\begin{corollary}\label{cor:h-sub-seq-inter}
There exists a subsequence $\mathcal{T}_{n_j}$ such that $\mathcal{T}_{n_j} \subset \mathcal{T}_{n_{j+1}}$.
\end{corollary}
The proof of Corollary \ref{cor:h-sub-seq-inter} is equal as the of Corollary \ref{cor:d-t-sub-seq-inter}. For simplicity, we will use the notation $\mathcal{T}_{n}$ for the subsequence $T_{n_j}$.
\begin{lemma}\label{prp:h-convergent-sequece}
There exist compact set $K_N \subset \R^{3}_F$ such that if $n > N$ then $c_{n}(t)$ is in $Min(K_N,\mathcal{T}_N)$.
\end{lemma}
The proof of Lemma \ref{prp:h-convergent-sequece} is equal to the of Lemma \ref{prp:d-t-convergent-sequece}.
Therefore, $\tilde{c}_{j}(t)$ has a convergent subsequence $\tilde{c}_{j_i}(t)$ converging to a $\R^3_{F_{h}}$-geodesic $c_{\infty}(t)$. Corollary \ref{cor:h-non-line-geo} implies that $c_{\infty}(t)$ is a normal $\R^3_{F_{h}}$-geodesic for a polynomial $G(x)$ in $Pen_{F_h}$.
The following Lemma tells $G(x) = F_{h}(x)$.
\begin{lemma}\label{prp:h-F-uniqueness}
$G(x) = F_{h}(x)$ is the unique polynomial in the pencil of $F_{h}(x)$ satisfying the asymptotic conditions given by \eqref{eq:h-asy-cond} and \eqref{eq:h-period-asy-cond}.
\end{lemma}

\begin{proof}
By Proposition \ref{prp:seque-comp} $\tilde{c}_{n}(t)$ has a convergent subsequence $\tilde{c}_{n_s}(t)$ converging to a minimizing geodesic $\tilde{c}(t)$ on the interval $\mathcal{T}_N$. Being a geodesic in $\R^3_{F_{h}}$, $c(t)$ is associated to a polynomial $G(x) = a +bF_{h}(x)$. $G(0)= a+b$ must be equal $1$, to satisfy the asymptotic conditions given by \eqref{eq:h-asy-cond}. Then $(a,b)$ is in $Pen_{h}^*$, the set defined in Corollary \ref{cor:h-The}. Since the map $\Theta_1(G,I):Pen_{h}^+ \to \R$ is one to one, the unique polynomial in $Pen_{h}^*$ satisfying the condition  \eqref{eq:h-asy-cond} is $G(x) = F_h(x)$. 
\end{proof}

The proof of Theorems \ref{th1:h-me-lin} and \ref{th:main-2} are the same as the proof of Theorems \ref{th1:d-t-me-lin} and \ref{th:main-1}, respectively.

\section{Conclusion}

We formalized the method used in \cite{RM-ABD} to prove that a particular geodesic is a metric line. Theorem \ref{th:main-1} proves the Conjecture \ref{con:1} for the heteroclinic of the direct-type case, and the problem remains open for the homoclinic case. Theorem \ref{th1:h-no-me-lin} says we cannot use the space $R^3_F$ to prove the Conjecture for the homoclinic case. However, Theorem \ref{th1:h-no-me-lin} does not imply that the Conjecture is false. The homoclinic case can be solved by showing the corresponding period map in $\Je$ restricted to the homoclinic geodesics is one-to-one.

\appendix

 \section{Proof Of Lemma \ref{lem:d-t-bounded-trav-interval}}\label{sub-sec:single-var}

\begin{definition}
Let $\mathcal{P}(k)$ be the vector space of polynomial on $\R$  of degree bounded by $k$, and let  $||F||_{\infty} :=  \sup_{x \in [0,1]} |F(x)|$ be the uniform norm. We denote by $B(k)$ the closed ball of radius 1.
\end{definition}

\begin{proposition}\label{prp:com-pol}
 $B(k)$ is a compact set. 
\end{proposition}

\begin{proof}
Since $B(k)$ is a bounded subset of the finite-dimensional space $\mathcal{P}(k)$, it is enough to prove that $B(k)$ is closed, indeed, by Arzela-Ascoli theorem we just need to prove that $B(k)$ is an equi-continuous set: let $F(x)$ be a polynomial in $C(k)$, then the Markov brothers' inequality implies $|F'(x)| \leq k^2$, so $|F(x_1)-F(x_2)| < k^2 |x_1-x_2|$.     
\end{proof}

\begin{definition}\label{def:normal-pol}
We say  a polynomial $F$ is unitary if $F$ has a hill interval $[0,1]$, and let $\mathcal{P}_N(k)$ be the set of unitary polynomials.    
\end{definition}

\begin{corollary}\label{cor:even-poly-F}
If $G_{n}(x)$ is a sequence of non-constant polynomials in $Pen_F$ with hill interval $I_n = [x_n,x'_n]$ such that $G_{n}(x_n) = G_{n}(x'_n) = 1$, $\lim_{n \to \infty} x_n = -\infty$ and $\lim_{n \to \infty} x'_n = \infty$, then $F(x)$ must be even degree.  
\end{corollary}

\begin{proof}
Let $G_n(x)$ be equal to $a_n + b_nF(x)$. There exists $K_x$ a compact set containing all the roots of $F(x)$, and let $n$ be large enough that $K_x \subset I_n$. Let us assume $F(x)$ is an odd degree. Without loss of generality, let us assume $F(x'_n) > 0$ and $ F(x_n) < 0$, then $0 = G(x'_n) - G(x_n) = b_n(F(x'_n)-F(x_n))$ and $b_n = 0$ since $F(x'_n)-F(x_n) > 0$, which is a contradiction to the assumption that $G_{n}(x)$ is a sequence of non-constant polynomials.
\end{proof}

\subsection{Proof Of Lemma \ref{lem:d-t-bounded-trav-interval}}\label{proof:bounded-trav-interval}
\begin{proof}
Let $c_{n}(t) = (x_n(t),y_n(t),z_n(t))$ be a sequence of $\R^{3}_{F}$-geodesics traveling during a time interval $[(t_0)_n,(t_1)_n]$  and with travel interval $\mathcal{I}_n [(t_0)_n,(t_1)_n]$ such that $x_n ((t_0)_n)$ and $x_n ((t_1)_n)$ are in $[0,1]$ for all $n$. We will prove that if $\mathcal{I}_n$ is unbounded, then $\Theta(c,[t_0,t_1])$ is unbounded.

The sequence of $c_{n}(t)$ of $\R^{3}_{F}$-geodesics, induces a sequence of $G_n(x)$ polynomials. We use the sequence $G_n(x)$ to define a sequence of unitary polynomials $\hat{G}_n(\tilde{x}) := G_n(h_n(\tilde{x}))$ where $h_n(\tilde{x}) = (x_0)_n + u_n \tilde{x}$ with $u_n := (x_0)_n - (x_1)_n$. Since $\hat{G}_n(\tilde{x}) $ is in $C(k)$. There exists a subsequence $\hat{G}_{n_j}(\tilde{x})$ converging to $\hat{G}(\tilde{x})$. Let us proceed by the following  cases: case  $\hat{G}(\tilde{x}) \neq 1$ or case $\hat{G}(\tilde{x}) = 1$.

Case  $\hat{G}(\tilde{x}) \neq 1$:  by Fatou's lemma $ 0 < Cost(\hat{G}) \leq \liminf_{n_j\to \infty} Cost(\hat{G}_n)$. Then $u_{n_j} \to \infty$ implies $Cost(c,\mathcal{I}_{n_j})$ is unbounded.

Case $\hat{G}(\tilde{x}) = 1$:  let $K_{x}'$ be a compact set such that all the roots of $1-F(x)$ are in $K_{x}'$. There exists $n^* >0$ such that $\hat{G}(\tilde{x}) > \frac{1}{2}$ for all $\tilde{x}$ in $[0,1]$ if $n_s > n^*$.  We split the integral for $\Delta z(c,\mathcal{I}_n)$ given by Corollary \ref{prop:mag-geo-Delta} in the following way  
\begin{equation}
\begin{split}
\int_{\mathcal{I}_{n_j}} \frac{ (1-F(x))G_{n_j}(x)}{\sqrt{1-G^2_{n_j}(x)}} dx  = & \int_{K'_{x}\cap \mathcal{I} } \frac{(1- F(x)) G_{n_j}(x)}{\sqrt{1-G^2_{n_j}(x)}} dx  \\
&  + \int_{(K_{x}')^c\cap \mathcal{I}} \frac{(1- F(x)) G_{n_j}(x)}{\sqrt{1-G^2_{n_j}(x)}} dx.
\end{split}
\end{equation}
Since the first integral of the right side is finite,  it is enough to focus on the second integral. 

We proceed by cases:  Case 1, $(x_0)_{n_j}$ and $(x_1)_{n_j}$ are both unbounded. Case 2,  $(x_0)_{n_j}$ is bounded and $(x_1)$ is unbounded. Case 3, $(x_0)_{n_j}$ is unbounded and $(x_1)_{n_j}$ is bounded.

Case 1: by Corollary \ref{cor:even-poly-F} implies $F(x)$ is even, then the condition $\hat{G}(\tilde{x}) > \frac{1}{2}$ implies $|G_{n_j}(x)| > \frac{1}{2}$ in the travel interval $\mathcal{I}_{n_j}$ and $(1-F(x))G_{n_j}(x)$ does not change sign in the set $\mathcal{I}_{n_j} \backslash K'_{x}$, therefore
$$ |\int_{\mathcal{I}_{n_j} \backslash K'_{x}} \frac{(1-F(x)) G_{n_j}(x) }{\sqrt{1-G^2_{n_j}(x)}} dx| >  \frac{1}{2 }\int_{\mathcal{I}_{n_j} \backslash K'_{x}} |F(x)| dx \to \infty \;\; \text{when}\;\; n_j \to \infty. $$
A similar proof follows for Cases 2 and 3.   
\end{proof}

\section{Proof Of Theorem \ref{th1:h-no-me-lin}}\label{sec:th1_:h-non-me-lin}

For simplicity, we will prove Theorem \ref{th1:h-no-me-lin} for the case $F(x) = 1-2x^3$. Let $c(t)$ be a $\R^{3}_{F}$-geodesic for $F(x) = 1-2x^3$  with initial point $c(0) = (1,0,0)$ and hill interval $[0,1]$. Let us consider the time interval $[-n,n]$, since the reduced system is  reservable it follows $x_n := x(n) = x(-n)$ and the travel interval $\mathcal{I}_n = 0[x_n,1]$.  By Corollary \ref{prop:mag-geo-Delta}, the relation between the travel interval and $n$ is given by
$$ n = \int_{[x_n,1]} \frac{dx}{\sqrt{1-F^2(x)}}. $$
In addition, the change in $\Delta y(c,n)$ and $\Delta z(c,n)$ are given by
$$ \Delta y(c,n) := 2\int_{[x_n,1]} \frac{F(x) dx}{\sqrt{1-F^2(x)}} \;\; \text{and} \;\; \Delta z(c,n) = 2\int_{[x_n,1]} \frac{F^2(x) dx}{\sqrt{1-F^2(x)}}.  $$
Therefore
$$ c(-n) = (x(-n),-\frac{\Delta y(F,n)}{2}, -\frac{\Delta z(F,n)}{2})$$
and
$$c(n) = (x(n),\frac{\Delta y(F,n)}{2}, \frac{\Delta z(F,n)}{2}). $$

\begin{corollary}
If $F(x) = 1-2x^3$ and $n$ is large enough, then 
$$\Delta y(F,n) < \Delta z(F,n) \;\;\text{and}\;\;  \lim_{n \to \infty} \frac{\Delta z(F,n)}{\Delta y(F,n)} = 1.$$
\end{corollary}
\begin{proof}
If $F(x) = 1-2x^3$, then the same integration by parts, used to prove Corollary \ref{lem:h-The}, implies the inequality $\Delta y(F,n) - \Delta z(F,n) > 0$. L'Hopital rules shows $\lim_{n\to \infty} \frac{\Delta z(F,n)}{\Delta y(F,n)} = 1$. 
\end{proof}

\begin{figure}%
    \centering
    {{\includegraphics[width=4cm]{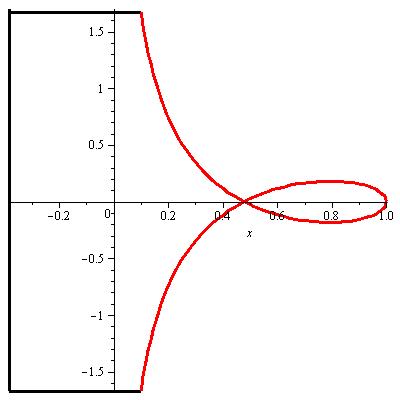} }}
    \qquad
    {{\includegraphics[width=4cm]{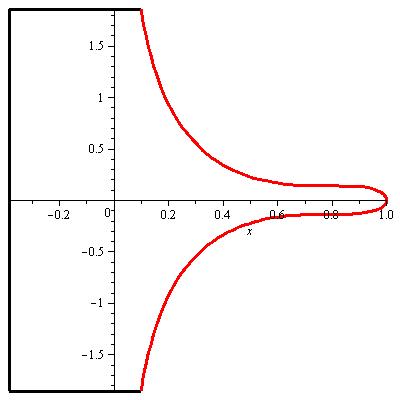} }}%
    \caption{Both images show the projection of the geodesic $c(t)$ for $F(x) = 1-2x^3$ and the curve $\tilde{c}(t)$ to the $(x,y)$ and $(x,z)$ planes, respectively.}    
    \label{fig:counter}
\end{figure}

\subsection{Proof Of Theorem \ref{th1:h-no-me-lin}}
\begin{proof}
For every large enough $n$ we can find $0 < \epsilon_n $ and $0 < \delta_n $ such that 
$$\Delta z(F,n) = (1+\epsilon_n)\Delta y(F,n)\;\; \text{and} \;\; F(-\delta_n) = 1+\epsilon_n.$$
If $T_1 = x_n + \delta_n $, $T_2 = T_1 + \Delta y(F,n)$ and $T_3 = T_1 + T_2$, then for every $n$ we define  the following curve $\tilde{c}_n(t)$ in $\R^{3}_{F}$ in the interval $[0,T_3]$ as follows
\begin{equation*}
\tilde{c}_n(t) = \begin{cases}
                c(-n) + (-t,0,0) \;\; \text{where}\;\; t \in [0,T_1] \\
                c(-n) + (-T_1,t-T_1,0) \;\; \text{where}\;\; t \in [T_1,T_2] \\
                c(-n) + (-T_1+t-T_2,\Delta y(F,n),\Delta z(F,n)) \; \text{where}\; t \in [T_2,T_3]. \\
                \end{cases}
\end{equation*}
See figure \ref{fig:counter}. By construction, $c(-n) = \tilde{c}_n(0)$ and $c(n) = \tilde{c}_n(T_3)$, the relation between the $n$ and $\Delta y(F,n)$ is given by 
$$2n = \Delta y(F,n) + Cost_t(F,[-n,n]),$$
while the relation between $T_3$ and $\Delta y(F,n)$ is given by 
$$T_3  =   \Delta y(F,n) + 2(\delta_n + x_n).$$
If $n \to \infty$, then $Cost_t(F,[-n,n]) \to \Theta_1(F,[0,1]) > 0$ and $2(\delta_n + x_n)\to 0$. Thus there exists an $n$ such that $Cost_t(F,[-n,n]) > 2(\delta_n + x_n)$, in other words $T_3 < 2n$ and we conclude that $\tilde{c}_n(t)$ is shorter that $c(t)$.

\end{proof}

\nocite{*} 
\bibliographystyle{plain}
\bibliography{bibli}

\begin{thebibliography}{10}

\bibitem{agrachev}
A.A Agrachev, D.~Barilari, and U.~Boscain.
\newblock {\em A Comprehensive Introduction to Sub-Riemannian Geometry}.
\newblock Cambridge University Press, 2019.

\bibitem{agrachev2004control}
A.A. Agrachev and Y.L. Sachkov.
\newblock {\em Control Theory from the Geometric Viewpoint}.
\newblock Control theory and optimization. Springer, 2004.

\bibitem{Monroy1}
A.~Anzaldo-Meneses and F.~Monroy-Perez.
\newblock Goursat distribution and sub-riemannian structures.
\newblock {\em Journal of Mathematical Physics}, 44(12):6101--6111.

\bibitem{Monroy2}
A.~Anzaldo-Meneses and F.~Monroy-Perez.
\newblock Integrability of nilpotent sub-riemannian structures.

\bibitem{ardentov2}
A.A. Ardentov and Y.L. Sachkov.
\newblock Conjugate points in nilpotent sub-riemannian problem on the engel
  group.
\newblock {\em Journal of Mathematical Sciences}, 195(3):369--390.

\bibitem{ardentov1}
A.A. Ardentov and Y.L. Sachkov.
\newblock Extremal trajectories in a nilpotent sub-riemannian problem on the
  engel group.
\newblock {\em Sbornik: Mathematics}, 202(11):1593--1615.

\bibitem{ardentov3}
A.A. Ardentov and Y.L. Sachkov.
\newblock Cut time in sub-riemannian problem on engel group.
\newblock 2014.

\bibitem{ardentov4}
A.A. Ardentov and Y.L. Sachkov.
\newblock Maxwell strata and cut locus in sub-riemannian problem on engel
  group.
\newblock {\em Regular and Chaotic Dynamics}, 22, 09 2017.

\bibitem{Arnold}
V.~I. Arnol'd.
\newblock {\em Mathematical methods of classical mechanics}.
\newblock Springer Science.

\bibitem{ABD}
A.~Bravo-Doddoli.
\newblock Higher elastica: Geodesics in the jet space.
\newblock {\em European Journal of Mathematics}.

\bibitem{ABD-No-perio}
A.~Bravo-Doddoli.
\newblock No periodic geodesic in the jet space.
\newblock {\em Pacific Journal of Mathematics}.

\bibitem{RM-ABD}
A.~Bravo-Doddoli and R.~Montgomery.
\newblock {Geodesics in Jet Space}.
\newblock {\em Regular and Chaotic Dynamics}, 2021.

\bibitem{BryantHsu}
R.~Bryant and L.~Hsu.
\newblock Rigidity of integral curves of rank 2 distributions.
\newblock {\em Inventiones mathematicae}, 114(2):435--462, 1993.

\bibitem{fathi2008weak}
A.~Fathi.
\newblock {\em The Weak KAM Theorem in Lagrangian Dynamics}.
\newblock Cambridge Studies in Advanced Mathematics. Cambridge University
  Press, 2008.

\bibitem{jurdjevic}
Velimir Jurdjevic.
\newblock {\em Geometric Control Theory}.
\newblock Cambridge Studies in Advanced Mathematics. Cambridge University
  Press, 1996.

\bibitem{Landau}
L.D. Landau and E.M. Lifshitz.
\newblock {\em Mechanics third edition: Volume 1 of course of theoretical
  physics}.
\newblock 1976.

\bibitem{Monroy3}
F.~Monroy-P{\'e}rez and A.~Anzaldo-Meneses.
\newblock Optimal control on nilpotent lie groups.
\newblock {\em Journal of Dynamical and Control Systems}, 8:487--504, 2002.

\bibitem{tour}
R.~Montgomery.
\newblock {\em A tour of subriemannian geometries, their geodesics and
  applications}.
\newblock Number~91. American Mathematical Soc.

\bibitem{monster}
R.~Montgomery and M.~Zhitomirskii.
\newblock Geometric approach to goursat flags.
\newblock volume~18, pages 459--493.

\bibitem{CarnotJets}
B.~Warhurst.
\newblock Jet spaces as nonrigid carnot groups.
\newblock {\em Journal of Lie Theory}, 15(1):341--356, 2005.

\end{thebibliography}

\end{document}